\documentclass[11pt]{article}

\usepackage{geometry}
\usepackage{mathtools}
\usepackage{bbm}
\usepackage{tikz}
\usetikzlibrary{arrows}
\usetikzlibrary{patterns}
\usetikzlibrary{hobby}
\usetikzlibrary{decorations.pathreplacing}
\usepackage{subcaption}

\usepackage{pgfplots}
\pgfplotsset{compat=1.18}
\usepgfplotslibrary{fillbetween}

\usepackage{graphicx}
\usepackage{amsmath,amssymb,amsthm}
\usepackage{enumerate}
\usepackage{hyperref}
\usepackage{eucal}
\usepackage[active]{srcltx}
\usepackage[T1]{fontenc}
\usepackage{color}
\usepackage[makeroom]{cancel}

\newtheorem{theo}{Theorem}[section]
\newtheorem{lem}[theo]{Lemma}

\newtheorem{cor}[theo]{Corollary}
\newtheorem{prop}[theo]{Proposition}

\newcommand{\eps}{\varepsilon}
\newcommand{\qtext}[1]{\quad\mbox{#1}\quad}
\newcommand{\qqtext}[1]{\qquad\mbox{#1}\qquad}

\newcommand{\rd}{\mathrm d}
\newcommand{\R}{\mathbb{R}}

\newcommand{\N}{\mathbb{N}}

\newcommand{\Q}{{\mathcal{Q}}}
\newcommand{\E}{{\mathcal{E}}}
\newcommand{\EE}{{\mathbb{E}}}

\newcommand{\D}{{\mathcal{D}}}
\renewcommand{\H}{{\mathcal{H}}}

\renewcommand{\t}{\theta}

\newcommand{\cone}{{\mathfrak C}}

\newcommand{\uL}{{\underline L}}
\newcommand{\uC}{{\underline C}}

\newcommand{\pf}{{}_{\#}}

\renewcommand{\P}{\mathcal{P}}
\renewcommand{\o}{\omega}
\newcommand{\g}{\gamma}

\newcommand{\PP}{\mathbb P}
\renewcommand{\O}{\Omega}
\newcommand{\G}{\Gamma}
\newcommand{\pO}{{\partial\Omega}}

\newcommand{\pD}{{\partial D}}
\newcommand{\bD}{{\bar D}}
\newcommand{\C}{\mathcal C}
\newcommand{\Leb}{{\mathrm{Leb}}}
\newcommand{\1}{\mathbbm{1}}
\newcommand{\ba}{{A}}

\def\dive{\operatorname{div}}

\numberwithin{equation}{section}
\begin{document}
\title{Large deviations for sticky-reflecting Brownian motion with boundary diffusion}
\date{}
\author{Jean-Baptiste Casteras, L\'eonard Monsaingeon, Luca Nenna}

\maketitle
\abstract{
We study a Schilder-type large deviation principle for sticky-reflected Brownian motion with boundary diffusion, both at the static and sample path level in the short-time limit.
A sharp transition for the rate function occurs, depending on whether the tangential boundary diffusion is faster or slower than in the interior of the domain.
The resulting intrinsic distance naturally gives rise to a novel optimal transport model, where motion and kinetic energy are treated differently in the interior and along the boundary.
}
\\
\bigskip

% \noindent

%%%%%%%%%%%%%%%%%%%%%%%%%%%%%%%%%%%%%%%%%%%%%%%%%%%%%%%%%%%%%%%%%%%%%%%%%%%%%%%%%%%%%%%%%%
\section{Introduction}
In its classical version, Schilder's theorem completely describes on the path level the short-time behaviour of Brownian motion with generator $\Q=\frac 12\Delta$.
More precisely, the slowed-down Wiener measure $R^\eps$ on $\O=C([0,1];\R^d)$, corresponding to the process with generator $\Q^\eps=\frac\eps 2\Delta$, satisfies in the small-time regime $\eps\to 0$ the Large Deviation Principle
\begin{equation}
\label{eq:LDP_schilder}
R^\eps
\underset{\eps\to 0}{\asymp}
\exp\left(-\frac 1\eps C(\o)\right),
\hspace{1cm}
C(\o)=\frac 12\int_0^1|\dot\omega_t|^2dt
\end{equation}
with rate function given by the kinetic energy.
The short-time behaviour can of course be analyzed for more general diffusion processes and this is interesting for many reasons, but let us just mention three in connection with analysis, optimal transport, and PDEs:
\begin{enumerate}
\item
Varadhan's short-time formula \cite{varadhan1967behavior}
$$
-\lim\limits_{t\searrow 0} t\log p_t(x,y)=\frac 12 d^2(x,y)
$$
relates the transition kernel $p_t(x,y)$ to the intrinsic distance $d^2(x,y)$ of a given diffusion process, which in turn can be used to establish upper and lower heat kernel estimates, Harnack inequalities, etc.
 \item
 The short-time behaviour is a key ingredient for the so-called \emph{Schrödinger problem}, which has been recognized over the past 20 years as an entropic approximation of Monge-Kantorovich optimal transport \cite{leonard2012schrodinger,leonard2014survey,mikami2004monge}.
 More precisely, given two distributions $\rho_0,\rho_1\in \P(\R^d)$ and writing $H(P|R)=\int \frac{\rd P}{\rd R}\log\left(\frac{\rd P}{\rd R}\right)\rd R$ for the relative entropy of $P\ll R$ with respect to a reference measure $R$, the Schrödinger bridge problem reads
\begin{equation}
 \label{eq:SP_eps}
  \boxed{\C^\eps(\rho_0,\rho_1)=\min\limits_{P\in \P(\O)}\Big\{ \eps H(P\,\vert\,R^\eps)
  \qtext{s.t.}
  P_0=\rho_0,P_1=\rho_1\Big\}.}
 \end{equation}
 Here and throughout $e_t(\o)=\o_t$ denotes the time-$t$ evaluation map, and $P_0={e_0}_\# P$ and $P_1={e_1}_\# P$ denote the marginals of the path measure $P$ at times $t=0,1$.
 Roughly speaking, as soon as the slowed-down measure satisfies the LDP $R^\eps\asymp \exp(-\frac 1\eps C(\o))$ the entropic problem Gamma-converges to
 $$
 \C=\Gamma-\lim\limits_{\eps\to 0}\C^\eps
 $$
 given by
 $$
\boxed{\C(\rho_0,\rho_1)=\min\limits_P\left\{ \int C(\o) P(\rd\o)
\qtext{s.t.}
P_0=\rho_0,P_1=\rho_1\right\},}
 $$
 see e.g. \cite{leonard2012schrodinger,bernton2022entropic,baradat2020small,benamou2019generalized,carlier2017convergence}.
 This is nothing but the Lagrangian version of the celebrated Benamou-Brenier Eulerian interpretation of Monge-Kantorovich optimal transport \cite{BB}.
 Moreover, writing
 \begin{equation}
 \label{eq:static_cost_from_dyn}
 c(x,y)=\min\limits_\o\big\{C(\o)\qtext{s.t.}\o(0)=x,\o(1)=y\big\}
 \end{equation}
 one also has the equivalent static Kantorovich formulation
 $$
\boxed{\C(\rho_0,\rho_1)=\min\limits_\pi\left\{ \int c(x,y)\pi(\rd x,\rd y)
\qtext{s.t.} \pi_x=\rho_0,\pi_y=\rho_1\right\}.}
 $$
 Hence the rate function $C(\o)$ in Schilder's theorem is deeply related to the optimal transport problem.
 \item
Since the works of Jordan, Kinderlehrer and Otto \cite{JKO98}, it is well-known that the canonical Fokker-Planck equation $\partial_t\rho_t=\Q^*\rho_t=\frac 12 \Delta\rho_t$ (with no-flux conditions on the boundary of the domain $D\subset\R^d$) can be interpreted as the gradient flow of the entropy $\H(\rho)=\int_D\rho\log\rho\, \rd x$ with respect to the Euclidean Wasserstein distance $W^2(\mu,\nu)=\C(\mu,\nu)=\min\limits_\pi\int |x-y|^2\pi(\rd x,\rd y)$.
Two ingredients connect here this macroscopic PDE with the microscopic (reflected) Brownian motion: the entropy $\H(\rho)=\H(\rho\vert\Leb)$ is computed relatively to the stationary, reversible Lebesgue measure, and the Wasserstein distance is built upon the underlying Euclidean cost $c(x,y)=|x-y|^2$ given by the intrinsic distance of the stochastic process.
In the series of works \cite{ADPZ,duong2013wasserstein,liero2017microscopic,mielke2014relation} it was understood that this macroscopic gradient flow structure can be justified at the microscopic level precisely from large deviation principles.
In particular \eqref{eq:LDP_schilder} plays a key role in the fist order expansion (Gamma-convergence) relating discrete-time rate functions on the one-hand, characterizing the hydrodynamic limit $N\to\infty$ for a large system of independent Brownian particles, and on the other hand minimizing movements/JKO schemes, a discrete-time characterization of the dissipative gradient flow structure of the macroscopic heat equation.
We refer e.g. to \cite[Theorem 3]{ADPZ} for a self-contained statement and more details, but let us just summarize by saying that it is the LDP \eqref{eq:LDP_schilder} that determines the correct Wasserstein dissipation mechanism and entropy functional, which in turn gives the Fokker-Planck equation as a canonical gradient flow.
See also \cite[\S 13.3 and Theorem 13.37]{FK06}.
 \end{enumerate}

All of this can be extended to cover more general models and processes (see e.g. \cite{patterson2019large,patterson2024variational,peletier2022jump,peletier2023cosh} to name just a few), but the analysis is always restricted to somehow ``smooth'' settings and requires a case-to-case adaptation.
In this work, motivated by the three questions above and in particular by the connection with Fokker-Planck PDEs, we take interest in the short-time behaviour of a particular non-smooth diffusion, the so-called \emph{Sticky-reflected Brownian Motion with boundary diffusion} (SBM in short).
Sparked by applications in interacting particle systems with boundary or zero-range interactions \cite{aurell2020behavior,konarovskyi2024reversible,grothaus2020overdamped}, among others, SBM has received renewed attention in the last decade \cite{bormannM2024functional,bormann2024functional,konarovskyi2021spectral,peskir2015boundary,engelbert2014stochastic,bou2020sticky}.
In a given domain $D\subset \R^d$ SBM can be roughly described as follows.
While in the interior, the process performs standard Brownian motion.
Upon hitting the boundary, SBM ``sticks'' there for a while and thereafter undergoes purely tangential diffusion along $\G=\pD$.
The process then eventually reenters the domain and resumes standard Brownian motion, and so on.
A delicate balance determines the jump rate from/to the boundary and involves the \emph{local time} at the boundary, see later on Section~\ref{sec:generalities} for rigorous details.
Accordingly, the Fokker-Planck equation takes the form of a system of coupled interior/boundary parabolic equations \eqref{eq:FPa}.

By simple scaling arguments one can always assume that the interior diffusion occurs with volatility $\sigma_\text{int}=1$, but a particularly important parameter in the model is the tangential diffusivity coefficient $a=\sigma^2_{\G}>0$.
We will show that a sharp transition takes place across $a=1$, which is the critical threshold for which interior and boundary diffusions match tangentially:
For $a>1$ motion along the boundary is preferred, the rate function in our LDP will see the impact of $a$ in a nontrivial and global fashion, and therefore the intrinsic distance $d_a(x,y)$ induces a specific geometry that strongly depends on $a>1$.
On the other hand for $a\leq 1$ interior motion becomes more favourable from an energetic perspective, the rate function will not depend on $a$ anymore, and one recovers instead a purely Euclidean scenario $d_a(x,y)\equiv|x-y|$ independently of $a<1$.

In \cite{casteras2024sticky} the first two authors showed that the Fokker-Planck equation for SBM is a gradient flow for the standard Wasserstein distance when $a=1$.
This is expected when interior and boundary diffusions match tangentially, so that the intrinsic distance of SBM coincides with the Euclidean one $d^2_a(x,y)=|x-y|^2$.
The case $a>1$ should lead to a true metric gradient flow (in the sense of \emph{curves of maximal slope} \cite{AGS}) with respect to the Wasserstein distance $W^2_a$ induced by the cost $c(x,y)=d^2_a(x,y)$ and will be the subject of a future article \cite{CMN}.
For $a<1$, since $d^2_a(x,y)=|x-y|^2$ is independent of $a$, the natural transportation distance is the Euclidean Wasserstein distance, so any Wasserstein gradient flow would not see $a$.
On the other hand the probabilistic Fokker-Planck equation clearly depends on $a$, see later on Section~\ref{sec:generalities} and in particular \eqref{eq:FPa}.
This means that the gradient flow, if any, cannot be a true metric one, and that a finer understanding of the infinitesimal structure is needed.
This will be investigated in a separate work \cite{BMRvR}.

Our goal here is threefold: 1) we study the LDP in its own right, 2) for $a<1$ we exhibit a very natural example of Sturm's paradigm \cite{sturm1997diffusion} when the diffusion process is not fully determined by its intrinsic distance, and 3) for $a>1$ the LDP gives, via the Gamma-convergence entropic-to-deterministic problems, gives the natural dissipation mechanism and yields a limiting, nonstandard optimal transport problem which is thereby justified and will therefore be used in subsequent works to study gradient-flows.
For full disclosure, we stress at this point that the rigorous analysis will only be carried in the particular case of half-spaces $D=\R^d_+$.
\\

The paper is organized as follows.
Section~\ref{sec:generalities} contains a detailed description of SBM and its basic properties.
In Section~\ref{sec:heuristics} we give a heuristic derivation of various levels of LDPs and formally identify the correct rate functions.
In Section~\ref{sec:transition_kernel} we explicitly compute the Markov transition kernel of SBM in half-spaces.
In Section~\ref{sec:static} we build on this explicit representation to derive a static LDP for the initial-terminal joint marginals $R^{\eps}_{0,1}\asymp\exp(-\frac 1\eps c(x,y))$, Theorem~\ref{theo:LDP_static}.
Section~\ref{sec:cost} contains an analysis of the cost and identifies it as a dynamical Lagrangian minimization $c(x,y)=\min \int_0^1L(\o_t,\dot\o_t)\rd t$, Theorem~\ref{theo:c_static_dynamic-geodesics}.
Finally, in Section~\ref{sec:dynamical_LDP} we show exponential tightness of the path measure and strengthen the static LDP to a full sample paths LDP $R^\eps_x\asymp\exp(-\frac 1\eps C_x(\o))$, Theorem~\ref{theo:LDP_dynamic}.

%%%%%%%%%%%%%%%%%%%%%%%%%%%%%%%%%%%%%%%%%%%%%%%%%%%%%%%%%%%%%%%%%%%%%%%%%%%%%%%%%%%%%%%%%%
\section{Sticky-reflected Brownian Motion with boundary diffusion}
\label{sec:generalities}

In its simplest form, SBM was originally studied by Feller \cite{feller1952parabolic} for the classification of boundary conditions for one-dimensional diffusions, and was later extended to include boundary diffusions in higher dimensions \cite{ikeda1961construction,takanobu1988existence,watanabe1971stochastic} including to Wentzell boundary conditions.
Let $D\subset\R^d$ be a smooth domain with boundary $\Gamma=\pD$.
We write $\G$ for objects and integrals intrinsically defined on $\G$, seen as a $(d-1)$-manifold of its own, while we keep the notation $\pD$ for the boundary trace of objects defined in the whole domain $\bar D$ (typically appearing when integration by parts is performed).

The Sticky-reflected Brownian Motion with boundary diffusion (SBM in short) is the $\bar D$-valued stochastic process with Feller generator
\begin{equation}
\label{eq:def_generator_L}
\Q \phi(x)=
 \begin{cases}
        \frac 12\Delta \phi(x) & \mbox{if }x\in D\\
        \frac a2\Delta_\G \phi(x)-\theta\partial_n \phi(x) & \mbox{if }x\in\pD
 \end{cases}
 \end{equation}
 and domain
 \begin{equation}
 \label{eq:def_domain_generator_L}
 \D(\Q)=\Big\{\phi\in C(\bD)\qtext{s.t.} \Q \phi\in C(\bD)\big\}.
 \end{equation}
 Here $n$ denotes the outer unit normal, and we write throughout $\nabla_\G,\dive_\G,\Delta_\G$ for the tangential gradient, divergence, and Laplace-Beltrami operators along $\G=\pD$.
 The tangential diffusion coefficient $a>0$ will play a crucial role and we sometimes emphasize the dependence by writing $\Q=\Q_a$.
 On the other hand the stickiness $\theta$ plays a very minor role and we omit the dependence.
 The reversible stationary measure is
 \begin{equation}
 \label{eq:def_stationary_measure}
 \mu(\rd x)=\rd x + \frac{1}{2\theta}\sigma(\rd x),
 \end{equation}
 where $\sigma=\H^{d-1}_\pD$ denotes the Lebesgue measure on the boundary.
 Note that $\mu$ depends on $\t$, but crucially not on $a$.
Integration by parts
 \begin{multline}
 \label{eq:pi_reversible_IBP}
\int_\bD \phi \Q\psi \,\rd \mu
=\int_D \phi\frac 12\Delta\psi\,\rd x + \int_\G \phi \left(\frac a2\Delta_\G \psi -\theta\partial_n\psi\right) \frac {1}{2\theta}\rd \sigma
\\
=\frac 12\left(-\int_D \nabla\phi\cdot\nabla \psi\,\rd x + \int_\pD\phi\partial_n \psi \,\rd \sigma\right)
- \left(
\frac a{4\t}\int_{\G}
\nabla_{\G} \phi\cdot \nabla_{\G} \psi\, \rd \sigma
+\frac 12\int_\G \phi\partial_n\psi\, \rd \sigma
\right)
\\
=-\frac 12\int_\O \nabla\phi\cdot\nabla \psi\,\rd x -\frac a{4\t}\int_{\G}
\nabla_{\G} \phi\cdot \nabla_{\G} \psi\, \rd \sigma
 \end{multline}
 shows that $\mu$ is indeed symmetric, and the corresponding Dirichlet form is
 $$
 \E_a(\phi)=
 \frac 12\int_D|\nabla\phi|^2\rd x
 +\frac{a}{4\t}\int_\pD |\nabla_\G\phi|^2\rd\sigma
 $$
 with domain
 $$
 \D(\E_a)=\Big\{\phi\in H^1(D)\qqtext{s.t.}\phi|_{\pD}\in H^1(\pD)\Big\}.
 $$
 (Here and throughout $\phi|_{\pD}$ stands for the boundary trace $\operatorname{tr}\phi=\phi|_{\pD}$ of $\phi\in H^1(D)$).
 It is shown in \cite{grothaus2017stochastic} that $\E_a$ is symmetric, regular, strongly local, and recurrent.
 We denote by $p_t(x,\rd y)$ the transition kernel with density $p_t(x,y)=p_t(y,x)$ w.r.t to $\mu$.
By \cite[thm. 3.15 and thm 3.17]{grothaus2017stochastic} for any $x\in \bD$ (including $x\in \G=\pD$) there is a unique path-measure $R_x$ on $\O=C([0,1];\bD)$ solving the martingale problem and represented through the SDE
\begin{equation}
\label{eq:SDE}
 \begin{cases}
  \rd X_t=
  \1_D(X_t)\rd B_t +\1_{\G}(X_t)\left[\sqrt a\rd B^{\G}_t-\t n(X_t)\rd t\right]\\
  \rd  B^{\G}_t=\pi(X_t)\circ\rd B_t\\
  X_0=x
 \end{cases}.
\end{equation}
Here
$$
\pi(x)=\operatorname{Id}-n(x)n^t(x)\in M_{d\times d}(\R)
$$
denotes the orthogonal projection on the tangent space $T_x\pD$ at a point $x\in \pD$ with outer normal $n(x)$.
The Stratonovich SDE $\rd  B^{\G}_t=\pi(X_t)\circ\rd B_t$ simply means that $B^{\G}_t$ is a Brownian motion on the boundary $\G=\pD$ with Laplace-Beltrami generator $\frac 12 \Delta_\G$, see \cite{hsu2002stochastic}.
Note that the full $d$-dimensional Brownian motion $B_t$ inside $D$ contains more information than the tangential one $B^\G_t$.
In \cite{engelbert2014stochastic} it is shown that even in one dimension $D=\R^+$ the SDE representing 1-dimensional SBM does not have strong solutions, and it is therefore natural to represent the boundary BM $B^\G$ in terms of the interior BM.
An important feature of \eqref{eq:SDE} is that it characterizes the local time at the boundary $L^\G_t$ as
$$
\rd L^\G_t= \t\1_\G(X_t)\rd t,
$$
see e.g. \cite[thm. 5]{engelbert2014stochastic} and \cite[\S IV.7]{ikeda2014stochastic}.
In particular, stickiness $\t>0$  yields non-trivial sojourn on the boundary with occupation time
\begin{equation*}
% \label{eq:O_L}
O_t=\int_0^t \1_\G(X_s)\rd s =\frac 1\t L^\G_t.
\end{equation*}

Finally, for the sake of completeness and also for future reference, let us  derive the Fokker-Planck equation for SBM.
The relevant laws $X_t\sim \rho_t\in \P(\bD)$ will as always be absolutely continuous w.r.t the stationary measure $\mu$ in \eqref{eq:def_stationary_measure}, hence we only consider measures $\rho\ll\mu$ and accordingly write
$$
\rho(\rd x)=u(x)\rd x + v(x)\sigma(\rd x).
$$
In order to compute the dual $\Q^*$ of \eqref{eq:def_generator_L} we take $\phi\in \D(\Q)$ and integrate by parts (assuming that $u,v$ are smooth enough)
\begin{multline*}
 \int_{\bD}\phi(x)[\Q^*\rho](\rd x)
 =\int_{\bD}[\Q\phi](x)\rho(\rd x)
 = \int_{D}\left[\frac 12\Delta \phi\right] u
 +\int_{\G}\left[\frac a2\Delta_\G \phi -\t\partial_n \phi\right]\, v
 \\
 =
 \frac 12\left(\int_D \phi \Delta u +\int_{\pD} u\partial_n \phi - \phi\partial_n u \right)
 +
 \left(\int _{\G}\phi\frac a2\Delta_\G v -\t v \partial_n \phi\right)
 \\
 =\int_D\frac 12\Delta u \,\phi
 +\int_\G \left(\frac a 2\Delta_\G v-\partial_n u\right)\phi +
 \int_\pD\left(\frac u2 -\t v\right)\partial_n \phi.
\end{multline*}
(all the integrals being implicitly computed with respect to the Lebesgue measures $\rd x,\rd\sigma$ on $D$ and $\G=\pD$, respectively)
As a consequence the abstract Fokker-Planck equation $\partial_t\rho_t =\Q^*\rho_t$ can be written in weak form (with again $\rho_t=u_t\rd x+v_t\rd \sigma$)
\begin{align*}
\forall\,\phi\in \D(\Q):\hspace{1cm}
\int_D \phi \partial_t u_t+\int_\G \phi \partial_t v_t
 &=
 \frac{d}{dt}\left(\int_D \phi u_t+\int_\G \phi v_t\right)
 \\
 &
 =
 \frac{d}{dt}\int_{\bD} \phi  \rho_t
 =\int_{\bD} (\Q\phi)\rho_t
 =\int_{\bD} \phi \Q^*\rho_t
%  =\int_{\bO}(Af) \rd\rho_t
 \\
 &=
 \int_D \frac 12 \Delta u_t \,\phi  +\int_\G \left(\frac a2\Delta_\G v_t-\frac 12\partial_n u_t\right)\phi
\\
 &  \hspace{4cm} + \int_\pO\left(\frac {u_t}2 -\t v_t\right)\partial_n\phi.
\end{align*}
Taking first $\phi\in C^\infty_c(\O)$ gives simply
$$
\partial_t u_t=\frac 12\Delta u_t
$$
in the interior, and we can simply cancel $\int_D \phi \partial_t u_t =\int_D \phi \frac 12 \Delta u_t$ in the previous equality.
This leaves
$$
\forall\,\phi\in \D(\Q):\hspace{2cm}
\int_\G \phi \partial_t v_t
=
\int_\G \left(\frac a2\Delta_\G v_t-\frac 12\partial_n u_t\right)\phi  + \int_\pO\left(\frac 12u_t -\t v_t\right)\partial_n\phi.
$$
Taking now any $\phi\in C^\infty(\G)$ and extending to $\phi\in C^\infty(\bD)$ with zero normal derivative such that $\phi\in \D(\Q)$ (this is no too difficult), we see now that $\int_\G \phi\partial_t v_t = \int_\G \left(\frac a2 \Delta_\G v_t-\partial_n u_t\right)\phi$ for all $\phi\in C^\infty(\G)$, meaning that
$$
\partial_t v_t=\frac a2\Delta_\G v_t-\frac 12\partial_n u_t.
$$
Subtracting again from the previous equality, we are finally left with
$$
\forall\,\phi\in \D(\Q):\hspace{2cm}
0=\int_{\pD}\left(\frac 12u_t-\t v_t\right)\partial_n \phi.
$$
It is relatively easy to check that the normal trace $\partial_n:\, \D(\Q)\to C(\G)$ is surjective, and therefore
$$
 \left.\frac 12u_t\right|_{\pD}=\t v_t \qquad\mbox{on the boundary }\pO.
$$
In view of \eqref{eq:def_stationary_measure} this simply means that the density $f_t(x)$ of $\rho_t(\rd x)=f_t(x)\mu(\rd x)$ with respect to the stationary, reversible measure $\mu$ is continuous up to the boundary, which should come as no surprise.
Summarizing, the Fokker-Planck equations is
\begin{equation}
\label{eq:FPa}
\partial_t\rho_t=\Q^*\rho_t
\qquad\iff\qquad
\begin{cases}
\rho_t=u_t\rd x+v_t\rd \sigma,
\\
  \partial_t u_t =\frac 12 \Delta u_t & \mbox{in }D,\\
  \partial_t v_t=\frac a2 \Delta_\G v_t-\frac 12\partial_n u_t& \mbox{in }\G,\\
  \frac 12 u_t=\t v_t& \mbox{on }\partial\Omega.
\end{cases}
\end{equation}

%%%%%%%%%%%%%%%%%%%%%%%%%%%%%%%%%%%%%%%%%%%%%%%%%%%%%%%%%%%%%%%%%%%%%%%
\section{Heuristics and technical obstructions}
\label{sec:heuristics}
The SBM process slowed-down on time-scale $\eps>0$ is described by
 $$
 R^\eps_x\coloneqq
 \;\text{\bf the path-measure with generator }\Q^\eps= \eps\Q
 $$
started at $x$.
One of the classical approaches to analyze the large deviations of the diffusion $R^\eps$ as $\eps\to 0$ is as follows, see e.g. \cite{FK06,kraaij2018large}.
Consider the Hamiltonian
$$
H^\eps \phi\coloneqq \eps e^{-\frac\phi\eps}\Q^\eps e^{\frac\phi\eps}
$$
and its associated Hamilton-Jacobi nonlinear semi-group $V^\eps(t)\phi\coloneqq\eps \log \EE_x e^{\frac 1\eps \phi(X^\eps_t)}$ satisfying
$$
\begin{cases}
 \frac d{dt}V^\eps(t)\phi = H^\eps V^\eps(t)\phi,\\
 V^\eps(0)\phi=\phi.
\end{cases}
$$
Convergence of the generators
\begin{equation}
\label{eq:CV_Heps_to_H}
\lim\limits_{\eps\to 0}H^\eps\phi= H\phi
\end{equation}
in some vague sense and for a large enough class of functions $\phi$ should in principle imply convergence of generated semi-groups
\begin{equation}
\label{eq:CV_Veps_to_V}
\lim\limits_{\eps\to 0}V^\eps(t)_\phi=V(t)\phi,
\end{equation}
where the limit $V$ is defined by
$$
\begin{cases}
 \frac d{dt}V(t)\phi = H V(t)\phi,\\
 V(0)\phi=\phi.
\end{cases}
$$
By standard Varadhan-Bryc arguments \cite[chapters 4.2 and 4.3]{dembo2009large} one expects the LDP to hold with dynamical rate function given by the abstract time-slicing
\begin{equation}
\label{eq:C_dyn_slice}
C(\o)
\coloneqq
\sup\limits_{|\tau_N|\to 0}\sup\limits_{\phi_1,\dots,\phi_N}
\sum\limits_{i=1}^N\phi_i(\o_{t_i})-V(t_i-t_{i-1})\phi_i(\o_{t_{i-1}}),
\end{equation}
where the supremum is taken over all partitions $\tau_N$ of the time interval $[0,1]$ such that $0=t_0<t_1<\dots<t_N=1$ with size $|\tau_N|=\max|t_{i+1}-t_i|\to 0$.
In order to retrieve a more tractable expression for this, assume moreover that the variational representation
$$
H\phi(x)
=H(x,\nabla\phi(x))
=\sup\limits_{q\in \R^d}\left\{\nabla \phi(x)\cdot q-L(x,q)\right\},
\hspace{1cm}x\in \bD
$$
holds for a sufficiently well-behaved Lagrangian $L$, i-e with $H(x,\cdot)=L^*(x,\cdot)$ convex conjugate of one another in the sense Fenchel-Legendre duality.
Then $H$ generates a Nisio semi-group $V(t)$, the dynamical cost \eqref{eq:C_dyn_slice} can be computed from usual stochastic control theory \cite{fleming2006controlled} as a Lagrangian action
\begin{equation}
\label{eq:C_slice=C_dyn}
C(\o)=\int_0^1L(\o_t,\dot\o_t)\rd t,
\end{equation}
and the diffusion $X^\eps$ finally satisfies the LDP
$$
R^\eps_x\asymp \exp\left(-\frac 1\eps C_x(\o)\right)
\qqtext{with rate function} C_x(\o)=\iota_{\{\o_0=x\}}+C(\o).
$$
Here
$$
\iota_{\{\o_0=x\}}=
\begin{cases}
 0 & \text{if }\o_0=x\\
 +\infty &\text{else}
\end{cases}
$$
encodes the initial condition $X_0=x$, but one can also cover general initial distributions $X_0^\eps\sim\rho_0^\eps$ as long as they satisfy a LDP $\rho_0^\eps(\rd x)\asymp \exp(-\frac 1\eps C_0(x))$.
By the contraction principle one therefore expects that the initial-terminal joint distribution $R^\eps_{01}=(e_0,e_1)\pf R^\eps_x\in \P(\bD^2)$ satisfies the LDP
$$
R^\eps_{01}\asymp \exp\left(-\frac 1\eps c(x,y)\right)
$$
with the static cost \eqref{eq:static_cost_from_dyn} obtained by Lagrangian minimization
$$
c(x,y)=\min\limits_\o \left\{\int_0^1L(\o_t,\dot\o_t)\rd t\qqtext{s.t.}\o_0=x,\o_1=y\right\}.
$$

This determines the static LDP from the dynamic one, but the analysis goes both ways: by simple scaling the LDP for $\{R^\eps_{01}\}_{\eps}$ immediately gives the LDP for $\{R^{\eps\tau}_{01}\}_{\eps>0}=\{R^\eps_{0\tau}\}_{\eps>0}$ for any fixed $\tau>0$.
One expects in general the collection of all such $(0,\tau)$-joint distributions to fully determine the path measure, hence by time slicing the sample path large deviation should follow from the static one.
This is actually the route that we will take.
\\

Let us try to make this all of this more explicit in our particular setting.
With the generator $\Q^\eps=\eps\Q$ given by \eqref{eq:def_generator_L} one easily computes
\begin{equation}
\label{eq:Heps}
H^\eps\phi(x)=
\begin{cases}
\frac 12 |\nabla\phi|^2+\frac \eps 2\Delta\phi & \text{if }x\in D\\
\frac a2 |\nabla_\G\phi|^2 + \eps\left(\frac a2\Delta_\G\phi - \t\partial_n\phi\right) & \text{if }x\in\pD
\end{cases},
\end{equation}
from which at least formally
\begin{equation}
\label{eq:H}
H^\eps\phi(x)\xrightarrow[\eps\to 0]{}
H\phi(x)=
\begin{cases}
\frac 12 |\nabla\phi|^2 & \text{if }x\in D\\
\frac a2 |\nabla_\G\phi|^2 & \text{if }x\in\pD
\end{cases}.
\end{equation}
The Legendre transform can be computed explicitly as
\begin{equation}
 \label{eq:L}
L(x,q)=\sup\limits_{p} \left\{q\cdot p - H(x,p)\right\}
=
\begin{cases}
\frac 12 |q|^2 & \text{if }x\in D\\
\frac 1{2a} |q|^2 & \text{if }x\in \pD\text{ and }q\in T_x\pD\\
+\infty & \text{else}
\end{cases}.
\end{equation}
It is straightforward to check that $L$ is jointly lower semicontinuous if and only if $a\geq 1$, and therefore the map
$$
\o\longmapsto C(\o)\coloneqq\int_0^1 L(\o_t,\dot\o_t)\rd t\quad\text{is l.s.c. if and only if }
a\geq 1
$$
for the natural weak $H^1$ convergence, see later on Theorem~\ref{theo:c_static_dynamic-geodesics}.
As usual, if a particular LDP holds for some rate function (here $C(\o)$) it also holds for its lower semi-continuous relaxation, so imposing the lower semi-continuity of good rate functions allows to speak unambiguously of \emph{the} rate function.
In our particular context we thus expect that \emph{the} LDP should hold in fact with $L$ replaced by its lower semi-continuous relaxation $\uL$, with effective rate function
\begin{equation}
\label{eq:LDP_lsc_relaxation}
\underline C_x(\o)
\coloneqq
\iota_{\{\o_0=x\}}+\uC(\o)=
\iota_{\{\o_0=x\}}+\int_0^1\uL(\o_t,\dot\o_t)\rd t.
\end{equation}
Here we see the phase transition already appearing very naturally:
For $a\geq 1$ we have $\uL=L$ and the induced cost $\underline C=C$ feels the influence of $a$ in a nontrivial fashion.
On the other hand it is not difficult to see that for $a<1$ the relaxation is always given by $\underline L=L\vert_{a=1}$, which is somehow a monotone relation with respect to the diffusion coefficient in the trivial sense that $1=\sup\{a,\,a<1\}$.
In that case the effective rate function $\underline C=C\vert_{a=1}$ does not depend on $a$.

In \cite{sturm1997diffusion} K.T. Sturm showed that, perhaps surprisingly, diffusion processes are not always fully characterized by their intrinsic distance.
More precisely, given any process $X$ with intrinsic distance $d^2(x,y)$, he constructs a second process $\tilde X$ with \emph{strictly smaller diffusion  coefficients} and yet sharing the same intrinsic distance $\tilde d=d$ with $X$.
This monotonicity echoes our lower relaxation of the Lagrangian: for any $a<1$ the SBM $X$ has intrinsic distance given by the Euclidean cost, and picking any $\tilde a<a$ gives a different diffusion $\tilde X$ with smaller coefficients but same distance.
Clearly $a=1$ is special in this sense, since it is the larger such $a$ for which one can pick $\tilde a<a$ in this way, and $a=1$ is therefore a maximal diffusion of some sort.
This is likely not a coincidence and will be investigated in a future work \cite{BMRvR}.

On a different note, the phase transition is also suggested from the boundary conditions, about which we deliberately remained vague so far.
By definition of the Hamiltonian, a function $\phi$ is in the domain $\D(H^\eps)$ of \eqref{eq:Heps} if and only if $\exp(\phi/\eps)$ is in the domain $\D(\Q^\eps)=\D(\Q)$, or equivalently if and only if the two expression in the right-hand side of \eqref{eq:Heps} match continuously at the boundary $\pD$.
At least formally for $\eps\to 0$ this suggests
\begin{equation}
 \label{eq:nablaphi_nablaphiG}
|\nabla\phi(x)|^2=a|\nabla_\G\phi(x)|^2
\qqtext{for}x\in \pD.
\end{equation}
\begin{itemize}
 \item
For $a>1$ this simply means that, at boundary points, $\nabla\phi$ makes an angle $\alpha=\alpha(a)$ with the normal $n(x)$ given by
$$
\sin^2\alpha=\frac 1a.
$$
This angle $\alpha$ will show up again later on, and we will prove that geodesics for the intrinsic distance always enter or exit the boundary with angle of incidence $\alpha$ (see Theorem~\ref{theo:c_static_dynamic-geodesics}).
Notice that for $a=1$ this corresponds to a right-angle condition $\alpha=\pi/2$, or equivalently to the Neumann boundary condition $\partial_n\phi=0$ along $\pD$.
This is exactly the classical no-flux condition for the velocity field $v=\nabla\phi$ driving Lagrangian particles in the Benamou-Brenier formulation \cite{BB} of optimal transport.
\item
On the other hand for $a<1$ the constraint \eqref{eq:nablaphi_nablaphiG} becomes unfeasible unless the \emph{whole} gradient vanishes, $\nabla\phi|_\pD =0$.
Formally this means that particles can no longer move along the boundary (not even tangentially), which seems surprising at first sight:
choosing $x,y$ close enough on a concave or flat portion of the boundary $\pD$, on expects that minimizing curves realizing $c(x,y)=\min\limits_\o\,\int_0^1L(\o_t,\dot\o_t)\rd t$ should actually remain supported along the boundary.
This apparent paradox is yet another effect of the lack of lower semi-continuity for $a<1$, and is simply resolved by the fact that no such minimizer actually exists.
Indeed, for any curve $(\o_t)_{t\in[0,1]}$ supported along the boundary, it is easy to construct a curve $\o'\approx \o$ very close to $\o$ with the same endpoints, but moving inside $D$ for all $t\in(0,1)$ and thus with a strictly lesser cost
$$
\int_0^1L(\o'_t,\dot\o'_t)
=\frac 12\int |\dot{\o}'_t|^2\rd t
<\frac 1{2a}\int |\dot{\o}_t|^2\rd t
=\int_0^1L(\o_t,\dot\o_t).
$$
This shows that curves moving along the boundary are exponentially unlikely, as far as a putative rate function is concerned in the LDP, hence those trajectories will asymptotically never be seen  by our sticky diffusion.
\end{itemize}
To summarize, the LDP is correctly captured by \eqref{eq:LDP_lsc_relaxation} with the relaxed Lagrangian $\underline L$, and the phase transition across $a=1$ is simply the threshold determining whether $L$ is l.s.c. or not.
\\

There are two main obstructions to turning this informal approach into rigorous analysis:
\begin{enumerate}
 \item
For $a<1$ the lack of \emph{joint} lower semi-continuity of $(x,q)\mapsto L(x,q)$ (mostly in the $x$ variable) immediately prevents any rigorous application of the above abstract arguments, i.e. the chain of implications \eqref{eq:CV_Heps_to_H}$\implies$\eqref{eq:CV_Veps_to_V}$\implies$\eqref{eq:C_slice=C_dyn} is no longer justified.
Thus the lack of semicontinuity is not just an (almost aesthetic) matter of unique selection of a rate function, but rather a key step that fails in the dynamical/Lagrangian representation of $C(\o)$.
 \item
 For $a>1$ the running cost $\o\mapsto\int_0^1L(\o_t,\dot\o_t)\rd t$ is l.s.c. but the argument requires no matter what a well-posedness of the Hamilton-Jacobi problem
 $$
 \partial_t\phi+H\phi=0
 $$
in the viscosity sense (more precisely, a resolvent estimate for $\operatorname{id}-\tau H$ for small $\tau>0$ based on maximum principles, see \cite{FK06,kraaij2018large}).
In our particular context this corresponds to nonstandard, coupled interior-boundary Hamilton-Jacobi equations
$$
\begin{cases}
 \partial_t\phi + \frac 12|\nabla\phi|=0 & \text{if }x\in D,\\
 \partial_t\phi + \frac a2|\nabla_\G\phi|=0 & \text{if }x\in \pD.
\end{cases}
$$
The recent theory in \cite{barles2023modern} might allow handling such delicate systems with discontinuous Hamiltonians, but for the sake of brevity we did not pursue in this direction.
\end{enumerate}
Both obstacles are somehow orthogonal to each other: for $a\leq 1$ the Hamilton-Jacobi problem is standard but the lower semi-continuity fails, while for $a\geq 1$ the Lagrangian is lower semi-continuous but the Hamilton-Jacobi system becomes difficult to cope with.
In order to offer self-contained and rigorous proofs we opted here for a compromise and chose to work on half-spaces $D=\R^d_+$, where explicit formulas can be leveraged and standard stochastic calculus is sufficient to carry over the whole analysis.
Of course, in smooth domains the boundary looks locally as $\R^d_+$:
Because SBM still belongs to the realm of Feller diffusions with quadratic variance $\EE_x(|X_t-x|^2)\sim t$, one should expect that our results in half-spaces can be leveraged to cover general domains following a ``local-to-global'' construction, see e.g. \cite[chapters 8 and 9]{bellaiche1981geodesiques}, but this falls out of the scope of this paper
%
%%%%%%%%%%%%%%%%%%%%%%%%%%%%%%%%%%%%%%%%%%%%%%%%%%%%%%%%%%%%%%%%%%%%%%%%%%%%%%
\section{Transition kernel in half-spaces}
\label{sec:transition_kernel}
From now on and unless otherwise specified we denote
$$
 D=\R^d_+=\R^+\times\R^{d-1}\qqtext{with}\pD\simeq\R^{d-1}
 $$
 and
 $$
 x=(x_1,x')\in D
 \qqtext{with}
 x_1\in\R^+,\,x'=(x_2\dots x_d)\in\R^{d-1}.
$$
We shall often speak of $x_1,x'$ as the horizontal and vertical coordinates, respectively.
In this particular setting the SBM process $X_t=(X^1_t,X'_t)$ takes values in $\bar D$ and \eqref{eq:SDE} takes the more explicit form
\begin{equation}
\label{eq:SDE_plane}
 \begin{cases}
  \rd X^1_t=
  \1_{\{X^1_t>0\}}\rd B^1_t +\t\1_{\{X^1_t=0\}}\rd t\\
  \rd  X'_t=\left[\sqrt a\1_{\{X^1_t=0\}}+\1_{\{X^1_t>0\}}\right]\rd B'_t\\
  X_0=x
 \end{cases},
\end{equation}
where $B=(B^1,B')$ is a standard $d$-dimensional Brownian motion.
The goal of this section is to explicitly compute the transition kernel $p_t(x,\rd y)$ of SBM in this simple planar setting, given by \eqref{eq:kernel_2d_pt} below.
\\

The key observation here is that the horizontal $X^1$ evolution is uncoupled from the vertical $X'$ motion and corresponds to the 1-dimensional Sticky-reflected Brownian Motion studied in \cite{engelbert2014stochastic}.
In particular the local time at the boundary of the full process is just given by the local time at the origin for the first component, which we simply denote
$$
L_t=L^\G_t(X)=L^0_t(X^1).
$$
By \cite[\S 3 and theorem 5]{engelbert2014stochastic} the first equation in \eqref{eq:SDE_plane} encodes among other things a relation
$$
\rd L^0_t(X^1)=\t \1_{\{X^1_t=0\}}\rd t=\t \rd O_t
$$
between the occupation time
$$
O_t
\coloneqq
\int_0^t\1_\G(X_s)\rd s
=
\int_0^t\1_{\{X^1_s=0\}}\rd s
\qquad \in [0,t]
$$
and local time $L_t$, whence
\begin{equation}
\label{eq:L=theta_O}
L_t=\t O_t
\qquad
\in [0,\t t].
\end{equation}
Note that in \eqref{eq:SDE_plane} the vertical diffusion $X'$ is only coupled to the horizontal motion of $X^1$ through the volatility $\sigma_t=\sigma(X^1_t)=\sqrt a\1_{\{X^1_t=0\}}+\1_{\{X^1_t>0\}}$, with $X^1$ independent of $B'$.
By the strong Markov property and
$$
\int_0^t\1_{\{X^1_s=0\}}\rd s=O_t
,
\hspace{1cm}
\int_0^t\1_{\{X^1_s>0\}}\rd s
=\int_0^t\left[1-\1_{\{X^1_s=0\}}\right]\rd s=t-O_t,
$$
we can integrate $\rd  X'_t=\left[\sqrt a\1_{\{X^1_t=0\}}+\1_{\{X^1_t>0\}}\right]\rd B'_t$ explicitly as
\begin{equation}
\label{eq:X'_distrib_sum_BM}
X'_t = \sqrt a \tilde B'_{O_t}+\hat B'_{t-O_t}
\end{equation}
almost surely, where $\tilde B',\hat B'$ are independent $(d-1)$-dimensional Brownian motions, also independent of $X^1$ and therefore of the occupation time $O_t$.

In order to eliminate the invariance under $\R^{d-1}$-vertical translation we always start from
$$
X_0=(X^1_0,X'_0)=(x_1,0)
\qqtext{for }x_1\geq 0.
$$
We first condition on $(X^1_t,L_t)=(z,l)$ as
\begin{multline*}
p_t(x,\rd y) = \PP_x(X_t\in \rd y)
=
\PP_x(X^1_t\in \rd y_1,X'_{t}\in \rd y')
\\
=\int_{\R^+}\int_0^{\theta t}
\PP_x(X^1_t\in \rd y_1,X'_{t}\in \rd y'\,\big\vert\, X^1_t=z,L_t=l)\,\PP_{x_1}(X^1_t\in \rd z,L_t\in \rd l)
\\
=\int_{\R^+}\int_0^{\theta t}
\PP_x\left(X^1_t\in \rd y_1,X'_{t}\in \rd y'\,\Big\vert\, X^1_t=z,O_t=\frac 1\theta l\right)\,\PP_{x_1}(X^1_t\in \rd z,L_t\in \rd l).
\end{multline*}
The point here is that the bivariate distribution $(X^1,L)$ for the 1-dimensional SBM in the conditioning can be determined explicitly.
More precisely, for $x_1,z\geq 0$ let $T_0$ denote the first hitting time at the origin for 1D-Brownian motion, denote by
$$
h(t,x_1)
\coloneqq \frac{|x_1|}{\sqrt{2\pi}t^{\frac 32}} e^{-\frac{|x_1|^2}{2t}}
\qquad t\geq 0
$$
its distribution (i-e $\PP_{x_1}(T_0\in \rd t)=h(t,x_1)\rd t$), and let
$$
g^0_t(x_1,z)
\coloneqq
\frac{1}{(2\pi t)^{\frac 12}}\left[e^{-\frac{|x_1-z|^2}{2t}} - e^{-\frac{|x_1+z|^2}{2t}}\right],
\qquad z\in \R^+
$$
be the kernel of Brownian motion before killing at the origin (i-e $\PP_{x_1}(B^1_t\in\rd z,T_0<t)=g^0_t(x_1,z)\rd z$).
By \cite[Theorem 2]{casteras2023trivariate} there holds
\begin{multline}
\label{eq:bivariate_x}
\PP_{x_1}(X^1_t\in\rd z,\,L_t\in \rd l)
=g^0_t(x_1,z)\,\rd z \delta_0(\rd l)
\\
+
\frac{1}{\theta}h\left(t-\frac l\theta,l+x_1\right)\delta_0(\rd z)\rd l
+
2 h\left(t-\frac l\theta,l+x_1+z\right) \rd z\,\rd l,
\end{multline}
with implicitly $z\geq 0$ and $l\in [0,\theta t]$ as in \eqref{eq:L=theta_O}.
On the other hand, writing
$$
g(t,z')=\frac{1}{(2\pi t)^{\frac{d-1}2}}\exp\left(-\frac{|z'|^2}{2 t}\right),
\qquad z'\in \R^{d-1}
$$
for the standard $(d-1)$-Gaussian, we can also compute
\begin{multline*}
\PP_x\left(X^1_t\in \rd y_1,X'_t\in \rd y'\,\Big\vert\, X^1_t=z,O_t=\frac 1\theta l\right)
=
\delta_{z}(\rd  y_1) \PP_{x'}\left(\sqrt a \tilde B'_{O_t}+\hat B'_{t-O_t}\in \rd y'\,\Big\vert\,O_t=\frac 1\theta l\right)
\\
=\delta_{z}(\rd  y_1)  \PP_{x'}\left(\sqrt a\tilde B'_{\frac{ l}\theta}+\hat B'_{t-\frac{l}{\theta}}\in \rd y'\right)
=
\delta_{z}(\rd  y_1)  \PP_{x'}\left(\tilde B'_{a\frac{ l}\theta}+\hat B'_{t-\frac{l}{\theta}}\in \rd y'\right)
\\
=
\delta_{z}(\rd  y_1) g\left(a\frac{ l}\theta+ t-\frac{l}\theta,y'-x'\right)\rd y',
\end{multline*}
because the sum of independent Gaussian variables $\tilde B+\hat B$ remains Gaussian with additive variance.
Putting
$$
\ba\coloneqq a-1\in (-1,\infty)
$$
and gathering everything, we end-up with
\begin{multline}
\label{eq:kernel_2d_ptt}
p_t(x,\rd y)=
g^0_t(x_1,y_1)g(t,y'-x')\,\rd y_1\rd y'
\\
+\frac 1\theta\left(\int_0^{\theta t} h\left(t-\frac l\theta,l+x_1\right)g\left(t+\ba\frac l\theta,y'-x'\right)\rd l\right) \delta_0(\rd  y_1)\rd y'
\\
+
2\left(\int_0^{\theta t} h\left(t-\frac l\theta,l+x_1+y_1\right)g\left(t+\ba\frac l\theta,y'-x'\right)\rd l\right)\rd y_1\rd y'.
\end{multline}
Owing to $L_t=\t O_t\in[0,\t t]$ the local time runs in $l\in [0,\t t]$ above, hence since $\ba>-1$ the effective time arguments $t-\frac l\t\geq 0$ and $t+\ba\frac l\t\geq  t-\frac l\t\geq 0$ in $h,g$ remain nonnegative as they should.
Similarly we emphasize that all the time arguments will be nonnegative in the sequel, and this will be implicit throughout without any further mention.
In our simple planar setting the stationary measure \eqref{eq:def_stationary_measure} decomposes as
$
\mu(\rd y)=\left(\rd y_1+\frac{1}{2\t}\delta_0(\rd y_1)\right)\rd y'
% \eqqcolon
% \mu_1(\rd x_1)\rd x'.
$.
Recalling that $g^0(x_1,0)=0$ for killed Brownian motion, \eqref{eq:kernel_2d_ptt} can be finally be written more compactly as

\begin{equation}
  \label{eq:kernel_2d_pt}
	\addtolength{\fboxsep}{2pt}
			\boxed{
				\begin{split}
					p_t(x,\rd y)
                        &=
                        \Bigg[
                        g^0_t(x_1,y_1)g(t,y'-x')
                        \\
                        &\hspace{2cm}+2\int_0^{\theta t} h\left(t-\frac l\theta,l+x_1+y_1\right)g\left(t+\ba\frac l\theta,y'-x'\right)\rd l \Bigg]\mu(\rd y).
				\end{split}
                  }
\end{equation}

%%%%%%%%%%%%%%%%%%%%%%%%%%%%%%%%%%%%%%%%%%%%%%%%%%%%%%%%%%%%%%%%%%%%%%%%%%%%%%%%%%%%
\section{Static Large Deviation Principle}
\label{sec:static}
As already discussed, the short time behaviour will be captured at the static level by the LDP for the slowed down transition kernels
\begin{equation}
\label{eq:def_rhoeps_2d}
\rho^{\eps}_x(\rd y)\coloneqq p_\eps(x,\rd y)\hspace{1cm}\in \P(\bar D),
\end{equation}
where $x\in\bD$ is fixed and $p_\eps(x,\rd y)$ is given by \eqref{eq:kernel_2d_pt} with $t=\eps$.
The main goal of this section is to establish
\begin{theo}
\label{theo:LDP_static}
Let $a>0$ and write for convenience $A=a-1$.
For any fixed $x=(x_1,x')\in\bar D=\R^+\times\R$ the sequence $\left\{ \rho^\eps_x\right\}_{\eps>0}\in \P(\bar D)$ from \eqref{eq:def_rhoeps_2d} satisfies the LDP
$$
\boxed{\rho_x^\eps(\rd y)\underset{\eps\to 0}{\asymp} \exp\left(-\frac 1\eps c(x,y)\right)}
$$
with good rate function $y\mapsto c(x,y)$ given by
\begin{align}
\label{eq:c_mu<0}
\mbox{if }a\leq 1:&
\hspace{1cm}   c(x,y)\coloneqq \frac 12 |x-y|^2
\medskip
\\
\label{eq:c_mu>0}
\mbox{if }a>1:
 & \hspace{1cm} c(x,y)\coloneqq
 \frac 12
\begin{cases}
|x-y|^2 & \mbox{if }(x,y)\in \mathfrak C,\\
\frac{1}{a}\left(\sqrt{\ba} |x_1+y_1| + |y'-x'|\right)^2&  \mbox{else},
\end{cases}
\end{align}
where for $A=a-1>0$ the ``cone'' $\mathfrak C$ is
\begin{equation}
\label{eq:def_cone}
 \cone\coloneqq
 \Bigg\{(x,y)\in \bar D^2:\qquad|y'-x'|\leq \frac{1}{\sqrt{\ba}}\Big[|x_1+y_1|+2\sqrt{a}\sqrt{x_1 y_1}\Big]\Bigg\}.
\end{equation}
Moreover $(x,y)\mapsto c(x,y)$ is symmetric and continuous over $\bar D^2$.
\end{theo}
\noindent
Note that $c$ strongly depends on $a$ (in a continuous way), but not on $\t$.
This might first come as a surprise, but we shall make a case later on that $c(x,y)=d^2(x,y)$ is truly the intrinsic distance for the SBM diffusion.
On the other hand the $\t$-``stickiness'' clearly appears as a lower order drift term in \eqref{eq:SDE}, and intrinsic distances classically tend to depend on the purely diffusive part of the process only.

For later purposes it will be convenient to define
\begin{equation}
\label{eq:def_cone(x)}
\cone(x)\coloneqq \{y\in \bar D:\quad (x,y)\in \cone\},
\end{equation}
which is represented graphically in Figure~\ref{fig:cones} below.
Note that $\cone(x)$ is really a cone if $x_1=0$ with aperture exactly $\sin^2\alpha=\frac 1a$, while it is rather a horizontal paraboloid if $x_1>0$.
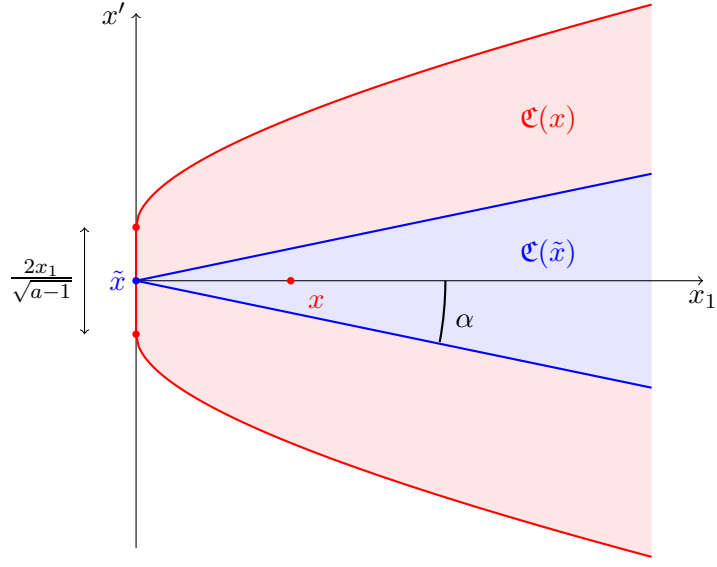
\begin{figure}[h!]
\begin{center}
\begin{tikzpicture}[scale=1.5]
\begin{axis}[ticks=none,axis lines=none,
xmin=-3, xmax=12,
ymin=-6, ymax=6
]
\draw [->] (axis cs:0,0) -- (axis cs:11,0);
\draw [->] (axis cs:0,-5) -- (axis cs:0,5);

\addplot [red,domain=0:10, samples=200, name path=f, thick]
{1+.1*x+sqrt(x)};
\addplot [red,domain=0:10, samples=200, name path=g, thick]
{-1-.1*x-sqrt(x)};
\draw [red,thick] (axis cs:0,-1) -- (axis cs:0,1);
\node at (axis cs:0,1)[red,circle,fill,inner sep=1pt]{};
\node at (axis cs:0,-1)[red,circle,fill,inner sep=1pt]{};
\draw [<->] (axis cs:-1,-1) -- (axis cs:-1,1);
\node [left] at (axis cs:-1,0) {$\frac{2x_1}{\sqrt{a-1}}$};

\addplot[red!10, opacity=0.4] fill between [of= f and g];
% \draw [red, thick] (axis cs:2,0) -- (axis cs:0,1) -- (axis cs:0,3) -- (axis cs:5,5);
% \node at (axis cs:0,1)[red,circle,fill,inner sep=1pt]{};
\node at (axis cs:3,0)[red,circle,fill,inner sep=1pt]{};
\node [red] at (axis cs:3.5,-.4) {$x$};
\node [red] at (axis cs:8,3) {$\mathfrak C(x)$};

\addplot [blue,domain=0:10, samples=200, name path=ff, thick]
{.2*x};
\addplot [blue,domain=0:10, samples=200, name path=gg, thick]
{-.2*x};
\addplot[blue!10, opacity=0.4] fill between [of= ff and gg];
\node at (axis cs:0,0)[blue,circle,fill,inner sep=1pt]{};
\node [blue,left] at (axis cs:0,0) {$\tilde x$};
\node [blue] at (axis cs:8,.5) {$\mathfrak C(\tilde x)$};

\node [below] at (axis cs:11,0) {$x_1$};
\node [left] at (axis cs:0,5) {$x'$};

\draw [thick] (axis cs:6,0) arc [radius=6,start angle=0,end angle=-11];
\node at (axis cs:6,-.75)[right]{$\alpha$};

\end{axis}
  \end{tikzpicture}
\end{center}
\caption{The cone $\mathfrak C(x)$.}
\label{fig:cones}
\end{figure}
It is worth stressing that the quantity $|x_1+y_1|$ appearing in\eqref{eq:c_mu>0}\eqref{eq:def_cone} can be naturally interpreted as a horizontal travel distance from $x_1$ to $y_1$ when forced to go through $z_1=0$.
This corresponds to a sticky scenario, where most random trajectories actually tend to go through the sticky boundary $z_1=0$ in order to move faster in the vertical direction by taking advantage of a stronger diffusion along the boundary if $a>1$.
Accordingly, a delicate balance will tip in favor of the ``going through the boundary'' strategy or not, depending on 1) the value of the diffusivity $a$, and 2) the relative location of the points $x,y$ to be connected (here $(x,y)\in \cone$ or not).
At a more technical level, let us anticipate that we will decompose below $\rho^\eps_x=\rho^\eps_\text{int}+\rho^\eps_\text{st}$ as the sum of an ``interior'' kernel (corresponding to paths avoiding the boundary) and of a ``sticky'' kernel (corresponding to paths with positive sojourn along $\pD$).
The interior part will always lead to a standard Euclidean contribution $\rho^\eps_\text{int}\asymp\exp(-\frac 1\eps I_\text{int}(x,y))$ with rate function $I_{\text{int}}=\frac 12|y-x|^2$.
The second part involves local time accounting for stickiness, and leads to a contribution $\rho^\eps_\text{st}\asymp\exp(-\frac 1\eps I_\text{st}(x,y))$ for a more complicated rate function $I_\text{st}$ that strongly feels the influence of the diffusion coefficient $a>0$.
Depending on the precise value of $a$ and relative positions of $x,y$, one or the other rate function will overtake, corresponding to the dichotomy in our statement.

In order to make the strategy of proof more precise we first change time scale
$$
l=\eps\theta L\qqtext{with}
L\in[0,1].
$$
From there \eqref{eq:kernel_2d_pt} conveniently leads to
\begin{align}
\rho^{\eps}_x(\rd y)
 & =
\Bigg[g^0_\eps(x_1,y_1)g(\eps,y'-x')\notag
\\
& \hspace{1cm} + \int_0^{\theta \eps} h\left(\eps-\frac l\theta,l+x_1+y_1\right)g\left(\eps+\ba\frac l\theta,y'-x'\right)\rd l\,\Bigg] \mu(\rd y)
\notag
\\
 & =
 \Bigg[ g^0_\eps(x_1,y_1)g(\eps,y'-x')\notag
\\
& \hspace{1cm}+   \int_0^1 h\left(\eps(1-L),\eps\theta L+x_1+y_1\right)g\left(\eps(1+\ba L),y'-x'\right)
 \,\eps\theta\rd L\Bigg] \mu(\rd y)\notag
\\
% & =
% \rho^\eps_\text{int}(x,y)\rd y+\left(\int_0^1 \bar \rho^\eps_\text{st}(x,y,L)\rd L\right)\mu(\rd y)\notag
% \\
& =
\Big[\rho^\eps_\text{int}(x,y)+\rho^\eps_\text{st}(x,y)\Big]\mu(\rd y),
\label{eq:def_Reps_2d}
\end{align}
where the interior and sticky densities are defined as
\begin{multline}
\label{eq:def_rho1}
\rho^\eps_\text{int}(x,y) \coloneqq g^0_\eps(x_1,y_1)g(\eps,y'-x')
\\
= \frac{1}{(2\pi \eps)^{d/2}}
\left[\exp\left(-\frac{|x_1-y_1|^2}{2\eps}\right)-\exp\left(-\frac{|x_1+y_1|^2}{2\eps}\right)\right]\exp\left(-\frac{|x'-y'|^2}{2\eps}\right)
\end{multline}
and
\begin{equation}
\label{eq:def_rho2_rhobar2}
\rho^\eps_\text{st}(x,y)\coloneqq \int_0^1\bar\rho^\eps_\text{st}(x,y,L)\rd L,
\end{equation}
with
\begin{multline}
 \label{eq:def_rho2}
\bar\rho^\eps_\text{st}(x,y,L)
\coloneqq \eps\theta h\left(\eps(1-L),\eps\theta L+x_1+y_1\right)g\left(\eps(1+\ba L),y'-x'\right)
\\
= \frac{\theta}{(2\pi\eps)^\frac d2}\times
\frac{\eps\theta L +x_1+y_1}{(1-L)^{\frac 32}(1+\ba L)^{\frac{d-1}2}}
\exp\left(-\frac{|\eps\theta L +x_1+y_1|^2}{2\eps(1-L)}\right)\exp\left(-\frac{|x'-y'|^2}{2\eps(1+\ba L)}\right).
\end{multline}

For the interior contribution \eqref{eq:def_rho1} one has $|x_1-y_1|<|x_1+y_1|$ for all $y_1$ (except for $y_1=0$ for which $g^0_\eps$ vanishes anyway), hence one expects $\exp\left(-\frac{|x_1-y_1|^2}{2\eps}\right)\gg \exp\left(-\frac{|x_1+y_1|^2}{2\eps}\right)$ and therefore one should anticipate
\begin{equation}
\label{eq:def_Iint}
\rho^\eps_\text{int}(x,y)
\asymp \exp\left(-\frac{1}{\eps}I_\text{int}(x,y)\right),
\hspace{1cm}
I_\text{int}(x,y)\coloneqq\frac 12 |x-y|^2.
\end{equation}
Viewing $\bar \rho^\eps_\text{st}(x,y,L)$ as acting on $L\in [0,1]$ as well, one reads off at least formally from \eqref{eq:def_rho2}
\begin{equation}
\label{eq:def_I2bar}
\bar \rho^\eps_\text{st}(x,y,L)
\asymp \exp\left(-\frac{1}{\eps}\bar I_\text{st}(x,y,L)\right),
\hspace{1cm} \bar I_\text{st}(x,y,L)\coloneqq \frac{|x_1+y_1|^2}{2(1-L)}+\frac{|x'-y'|^2}{2(1+\ba L)}.
\end{equation}
The contraction principle \cite[thm. 4.2.1]{dembo2009large} thus suggests
\begin{equation}
\label{eq:def_I2}
\rho^\eps_\text{st}(x,y)=\int_0^1\bar \rho^\eps_\text{st}(x,y,L)\rd L\asymp \exp\left(-\frac{I_\text{st}(x,y)}{\eps}\right),
\hspace{1cm} I_\text{st}(x,y)\coloneqq \min\limits_{L\in[0,1]}\bar I_\text{st}(x,y,L).
\end{equation}
Finally, the sum should at least formally satisfy an LDP with rate given by the usual rule of the ``least unlikely of the unlikely'', here
\begin{multline*}
 \rho^\eps_x(\rd y)=\Big[\rho^\eps_\text{int}(x,y) + \rho^\eps_\text{st}(x,y) \Big]\mu(\rd y)
\asymp \exp\left(-\frac{1}{\eps}c(x,y)\right)
\\
\text{with}\hspace{1cm}
c(x,y)=\min\left\{I_\text{int}(x,y),I_\text{st}(x,y)\right\}.
\end{multline*}
It turns out that for $a\leq 1\Leftrightarrow \ba\leq 0$ one always has $I_\text{int}\leq I_\text{st}$, while for $a >1$ one has $I_\text{st}<I_\text{int}$ if and only if $(x,y)\in \cone$ exactly as in \eqref{eq:c_mu<0}\eqref{eq:c_mu>0}.
This line of thought really gives the correct result, but requires special technical care due to the $\frac{1}{1-L}$ singularity appearing both in the exponential rate and multiplicative prefactor in \eqref{eq:def_rho2} and, to a lesser extent, to the atom at $y_1=0$ in $\mu(\rd y)$.
\\

Let us now make this sketch of proof rigorous.
For the sake of exposition we first establish two technical lemmas in order to fully determine $I_\text{int},I_\text{st}$ and $c=\min\{I_\text{int},I_\text{st}\}$ as functions of $x,y$.
\begin{lem}
 \label{lem:I2_epxlicit}
 Let $a>0$ and $\ba=a-1$.
 The sticky rate $I_\text{st}$ defined by \eqref{eq:def_I2bar}\eqref{eq:def_I2} reads
 \begin{equation}
 \label{eq:I_st_explicit}
 I_\text{st}(x,y)=\frac 12
 \begin{cases}
 |x_1+y_1|^2+|x'-y'|^2 & \mbox{if }a\leq 1\\
|x_1+y_1|^2+|x'-y'|^2 & \mbox{if }a> 1\mbox{ and }|y'-x'|\leq \frac 1{\sqrt{\ba}}|x_1+y_1|\\
\frac{1}{a}\left(\sqrt{\ba} |x_1+y_1|+|x'-y'|\right)^2 & \mbox{if }a> 1\mbox{ and }|y'-x'|>\frac 1{\sqrt{\ba}}|x_1+y_1|
 \end{cases}
\end{equation}
and is continuous in $y$.
\end{lem}
\begin{proof}
For fixed $x,y$ (with $x_1,y_1\geq 0$) let
$$
f(L)\coloneqq \frac{|x_1+y_1|^2}{1-L}+\frac{|x'-y'|^2}{1+\ba L},
$$
so that by definition $I_\text{st}(x,y)=\frac 12 \min\limits_{L\in[0,1]}f(L)$ in \eqref{eq:def_I2}.

In the easy case $\ba\leq 0\iff a\leq 1$ the function $f$ is monotone nondecreasing and clearly the minimum is attained for $L=0$ with value $I_\text{st}(x,y)=\frac 12 f(0)=\frac{1}{2}(|x_1+y_1|^2+|x'-y'|^2)$.

Consider now the case $\ba>0$, and observe that $f$ is convex.
If $f'(0)=|x_1+y_1|^2 -\ba|x'-y'|^2\geq 0$ then by convexity $f$ is minimized again at $L=0$ and leads to the exact same value as above.
This is exactly the second alternative in our statement.

Let us finally deal with the last case $f'(0)<0$, which is exactly our third alternative in \eqref{eq:I_st_explicit}.
Assume that $x_1+y_1>0$: then $f(1^-)=+\infty$, and by strict convexity $f$ is minimized for a unique $L^*\in(0,1)$.
A tedious but straightforward computation allows to characterize this critical point $f'(L^*)=0$ as the unique positive root of a certain quadratic polynomial in $L$ and leads to the precise value $f(L^*)=\frac{1}{a}\left(\sqrt\ba |x_1+y_1|+|x'-y'|\right)^2$ as in our statement.
It remains to observe that the very same formula remains valid also if $x_1+y_1=0$, which is a particularly easy case when $f(L)=\frac{|x'-y'|^2}{1+\ba L}$ is obviously minimized by $f(1)=\frac{|x'-y'|^2}{1+\ba}=\frac{|x'-y'|^2}{a}$.

Those three cases are illustrated in Figure~\ref{fig:cones_I2}.

Finally, the continuity of $y\mapsto I_\text{st}(x,y)$ is easy to check in view of the explicit formula \eqref{eq:I_st_explicit} and the proof is complete.
\end{proof}
\begin{lem}
\label{lem:min_I1_I2}
The cost
$$
c(x,y)\coloneqq \min\{I_\text{int}(x,y),I_\text{st}(x,y)\}
$$
is given explicitly by \eqref{eq:c_mu<0}\eqref{eq:c_mu>0}.
\end{lem}
\begin{proof}
Recall that by definition $I_\text{int}(x,y)=\frac 12 |x-y|^2$, so we need to compare this Euclidean cost to each of the three explicit cases in Lemma~\ref{lem:I2_epxlicit} to determine which of $I_\text{int},I_\text{st}$ realizes the minimum.
\begin{itemize}
 \item
If $a\leq 1$ we have by \eqref{eq:I_st_explicit} that $I_\text{int}(x,y)=\frac{1}{2}(|x_1-y_1|^2+|x'-y'|^2)\leq \frac{1}{2}(|x_1+y_1|^2+|x'-y'|^2)=I_\text{st}(x,y)$ due to $x_1,y_1\geq 0\implies|x_1-y_1|\leq |x_1+y_1|$.
\item
For the very same reason the second alternative in \eqref{eq:I_st_explicit} also leads to $I_\text{int}\leq I_\text{st}$.
\item
In the third case, straightforward algebra shows that
$$
I_\text{st}(x,y)\leq I_\text{int}(x,y)
\iff
\frac{1}{a}\left(\sqrt\ba |x_1+y_1|+|x'-y'|\right)^2  \leq |x_1-y_1|^2+|x'-y'|^2.
$$
Solving explicitly in $|x'-y'|$ shows that this quadratic inequality holds true outside of the two roots
$$
|y'-x'|\not\in [R^-,R^+],
\hspace{2cm}R^\pm=\frac{1}{\sqrt \ba}|x_1+y_1|\pm 2\sqrt{\frac{a}{\ba}}\sqrt{x_1y_1}.
$$
In particular our standing assumption that $|y'-x'|>\frac 1{\sqrt{\ba}}|x_1+y_1|$ from \eqref{eq:I_st_explicit} rules out the lower part $|y'-x'|\leq R^-$, hence in that case $I_\text{st}\leq I_\text{int}$ if and only if $|y'-x'|\geq R^+$.
According to \eqref{eq:def_cone} is exactly the definition of $(x,y)\not\in\cone$, or equivalently $y\not\in\cone(x)$, and the proof is complete.
\end{itemize}
Those three cases are again depicted in Figure~\ref{fig:cones_I2}.
Note that if $x_1>0$ the cone $\cone(x)$ is a true paraboloid as in Figure~\ref{fig:cones_I2_x1>0}, while if $x_1=0$ it collapses into a linear cone as in Figure~\ref{fig:cones_I2_x1=0}.
\end{proof}

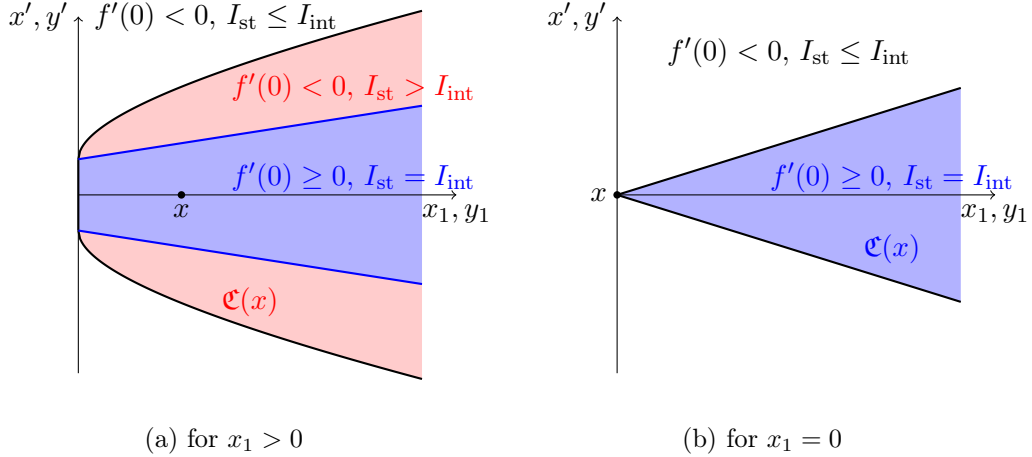
\begin{figure}[h!]

\begin{center}
\begin{subfigure}{0.4\textwidth}
\begin{tikzpicture}[]
\begin{axis}[ticks=none,axis lines=none,
xmin=-3, xmax=12,
ymin=-6, ymax=6
]
\draw [->] (axis cs:0,0) -- (axis cs:11,0);
\draw [->] (axis cs:0,-5) -- (axis cs:0,5);

\addplot [domain=0:10, samples=200, name path=f, thick]
{1+.1*x+sqrt(x)};
\addplot [domain=0:10, samples=200, name path=g, thick]
{-1-.1*x-sqrt(x)};
\draw [thick] (axis cs:0,-1) -- (axis cs:0,1);
% \node at (axis cs:0,1)[red,circle,fill,inner sep=1pt]{};
% \node at (axis cs:0,-1)[red,circle,fill,inner sep=1pt]{};
% \draw [<->] (axis cs:-1.5,-1) -- (axis cs:-1.5,1);
% \node [left] at (axis cs:-1.5,0) {$\frac{2x_1}{\sqrt\mu}$};

\addplot[red!20, opacity=0.4] fill between [of= f and g];
% \draw [red, thick] (axis cs:2,0) -- (axis cs:0,1) -- (axis cs:0,3) -- (axis cs:5,5);
% \node at (axis cs:0,1)[red,circle,fill,inner sep=1pt]{};
% \node at (axis cs:3,0)[circle,fill,inner sep=1pt]{};
% \node  at (axis cs:3.5,-.4) {$x$};
\node [red] at (axis cs:8,3) {$f'(0)<0,\, I_\text{st}> I_\text{int}$};
\node [   ] at (axis cs:4,5) {$f'(0)<0,\,I_\text{st}\leq  I_\text{int}$};

\addplot [blue,domain=0:10, samples=200, name path=ff, thick]
{1+.15*x};
\addplot [blue,domain=0:10, samples=200, name path=gg, thick]
{-1-.15*x};
\addplot[blue!30, opacity=0.4] fill between [of= ff and gg];
% \draw [red, thick] (axis cs:2,0) -- (axis cs:0,1) -- (axis cs:0,3) -- (axis cs:5,5);
% \node at (axis cs:0,1)[red,circle,fill,inner sep=1pt]{};
% \node at (axis cs:3,0)[red,circle,fill,inner sep=1pt]{};
\node [below] at (axis cs:3,0) {$x$};
\node at (axis cs:3,0)[circle,fill,inner sep=1pt]{};
\node [blue] at (axis cs:8,.5) {$f'(0)\geq 0,\, I_\text{st}=I_\text{int}$};

\node [below] at (axis cs:11,0) {$x_1,y_1$};
\node [left] at (axis cs:0,5) {$x',y'$};

\node [red] at (axis cs:5,-3) {$\cone(x)$};

\end{axis}

% \draw[help lines,step=1] (0,-5) grid (10,5);

% \addplot+[mark=none,samples=200,unbounded coords=jump] {sqrt(x)};
%     \node  at (-2,2) {$\O$};
%     \node  at (-.5,2.5) {$\G$};

%     \node [blue] at (-4,0) {$\supp\check\o_n$};
  \end{tikzpicture}
  \caption{for $x_1>0$}
  \label{fig:cones_I2_x1>0}
  \end{subfigure}
%   \hfill
\hspace{.3cm}
  \begin{subfigure}{0.4\textwidth}
\begin{tikzpicture}[]
\begin{axis}[ticks=none,axis lines=none,
xmin=-3, xmax=12,
ymin=-6, ymax=6
]
\draw [->] (axis cs:0,0) -- (axis cs:11,0);
\draw [->] (axis cs:0,-5) -- (axis cs:0,5);
\node [   ] at (axis cs:5,4) {$f'(0)<0,\,I_\text{st}\leq  I_\text{int}$};

\addplot [domain=0:10, samples=200, name path=ff, thick]
{.3*x};
\addplot [domain=0:10, samples=200, name path=gg, thick]
{-.3*x};
\addplot[blue!30, opacity=0.4] fill between [of= ff and gg];
% \draw [red, thick] (axis cs:2,0) -- (axis cs:0,1) -- (axis cs:0,3) -- (axis cs:5,5);
% \node at (axis cs:0,1)[red,circle,fill,inner sep=1pt]{};
% \node at (axis cs:3,0)[red,circle,fill,inner sep=1pt]{};
\node [left] at (axis cs:0,0) {$x$};
\node at (axis cs:0,0)[circle,fill,inner sep=1pt]{};
\node [blue] at (axis cs:8,.5) {$f'(0)\geq 0,\, I_\text{st}=I_\text{int}$};

\node [below] at (axis cs:11,0) {$x_1,y_1$};
\node [left] at (axis cs:0,5) {$x',y'$};

\node [blue] at (axis cs:8,-1.5) {$\cone(x)$};

\end{axis}
  \end{tikzpicture}
  \caption{for $x_1=0$}
  \label{fig:cones_I2_x1=0}
  \end{subfigure}
  \caption{The alternatives in Lemma~\ref{lem:I2_epxlicit} and Lemma~\ref{lem:min_I1_I2} (for $a>1$).}
  \label{fig:cones_I2}
\end{center}
\end{figure}
\noindent
We can now proceed with our main objective in this section
\begin{proof}[Proof of Theorem~\ref{theo:LDP_static}]
Note first that the cone $\cone$ is completely symmetric in $x\leftrightarrow y$, as is obvious from \eqref{eq:def_cone}.
Since the two expressions $|x-y|^2$ and $\frac{1}{a}\left(\sqrt\ba |x_1+y_1| + |y'-x'|\right)^2$ in \eqref{eq:c_mu<0}\eqref{eq:c_mu>0} are also symmetric, clearly $c(x,y)=c(y,x)$.
The continuity of $ c$ is immediate from its explicit expression, and it is also easy to check that it has compact sublevelsets and is thus a good rate function as claimed.

Let us now address the main difficulty, namely the LDP itself.
To this end, pick a Borel set $E\subset \bar D$ and assume $\mu(E)>0$ (otherwise $\rho^\eps_x\ll\mu$ implies $\rho^\eps_x(E)=0$ for all $\eps>0$ and there is nothing to prove).
\\

\medskip
\noindent
% % % % % % % % % % % % % % % % %
\textbf{\underline{Step 1: LDP upper bound.}}
Since $y\mapsto c(x,y)$ is coercive
$$
{\hat\lambda}\coloneqq \inf\limits_{y\in \bar E}c(x,y)
=\min\limits_{y\in \bar E}c(x,y)
\hspace{1cm}\in [0,\infty)
$$
if finite.
From \eqref{eq:def_Reps_2d} it suffices to estimate from above $\rho^\eps_\text{int},\rho^\eps_\text{st}\lesssim e^{-\frac {\hat\lambda}\eps}$ separately.

We start with the ``interior'' contribution.
Recall that $\mu(\rd y)=(\rd y_1+\frac 1{2\t}\delta_0(\rd y_1))\rd y'$, so $E$ is possibly concentrated on the boundary $\pD=\{y_1=0\}$ even if $\mu(E)>0$.
By definition of $c=\min\{I_\text{int},I_\text{st}\}$ and $\hat\lambda$ we have that $I_\text{int}(x,y)\geq c(x,y)\geq {\hat\lambda}$ for $\mu$-a.e. $y\in E$ (and even if $E$ concentrates on $\{y_1=0\}$).
Recalling also that $g^0_\eps(x_1,0)=0$ we then estimate from \eqref{eq:def_rho1}
\begin{multline*}
\int_{E}\rho^\eps_\text{int}(x,y)\mu(\rd y)
=
\int_{E}\rho^\eps_\text{int}(x,y)\rd y
\\
\leq
\frac{\theta}{(2\pi\eps)^{\frac d2} }\int_{E} \left[\exp\left(-\frac{|x_1-y_1|^2}{2\eps}\right)-0\right]
\exp\left(-\frac{|x'-y'|^2}{2\eps}\right)\rd y
\\
=
\frac{C}{\eps^{\frac d2}}\int_{ E} \exp\left(-\frac{I_\text{int}(x,y)}{\eps}\right) \rd y
=
\frac{C}{\eps^{\frac d2}}e^{-\frac{\hat\lambda}\eps}\int_{ E} \exp\left(-\frac{I_\text{int}(x,y)-{\hat\lambda}}{\eps}\right) \rd y.
\end{multline*}
By definition of ${\hat\lambda}$ we can simply write $0\leq I_\text{int}(x,y)-{\hat\lambda}\leq \frac{I_\text{int}(x,y)-{\hat\lambda}}{\eps}$ in the exponential, whence
\begin{multline}
\label{eq:upper_estimate_rho_int}
\int_{E}\rho^\eps_\text{int}(x,y)\mu(\rd y)
\leq
\frac{C}{\eps^{\frac d2}}e^{-\frac{\hat\lambda}\eps}\int_{ E} \exp\left(-[I_\text{int}(x,y)-{\hat\lambda}]\right) \rd y
\\
=
\frac{C}{\eps^{\frac d2}}e^{-\frac{\hat\lambda}\eps}\int_{ E} \exp\left(-\frac{|x-y|^2}{2}+{\hat\lambda}\right) \rd y
\\
\leq
\frac{C}{\eps^{\frac d2}}e^{-\frac{\hat\lambda}\eps}\int_{ \R^d_+} \exp\left(-\frac{|x-y|^2}{2}+{\hat\lambda}\right) \rd y
=
\frac{C}{\eps^{\frac d2}}e^{-\frac{\hat\lambda}\eps},
\end{multline}
where $C\geq 0$ varies from line to line and depends on $x,\theta,{\hat\lambda}, E$ but not on $\eps$.
Note that at this stage we have possibly $C=0$ if $E$ actually concentrates on the boundary.
This settles the upper estimate for the ``interior'' part in \eqref{eq:def_Reps_2d}.

For the second ``sticky'' contribution, recall that $\bar I_\text{st}(x,y,L)=\frac{|x_1+y_1|^2}{2(1-L)}+\frac{|x'-y'|^2}{2(1+\ba L)}$ and by definition $I_\text{st}(x,y)=\min\limits_{L\in[0,1]}\bar I_\text{st}(x,y,L)$.
Since ${\hat\lambda} \leq c(x,y)\leq I_\text{st}(x,y)$ for $\mu$-a.e. $y\in E$, we simply have $\bar I_\text{st}(x,y,L)\geq {\hat\lambda}$ for $\mu(\rd y)\rd L$-a.e. $(y,L)\in F\coloneqq E\times[0,1]$.
From \eqref{eq:def_rho2} we get
\begin{multline*}
\int_E\rho^\eps_\text{st}(x,y)\mu(\rd y)=
\int_{ E}\left(\int_0^1\bar \rho^\eps_\text{st}(x,y,L)\rd L\right)\mu(\rd y)
\\
=\frac{\theta}{(2\pi\eps)^{\frac d2}}
\int_{ F}
\frac{\eps\theta L +x_1+y_1}{(1-L)^{3/2}(1+\ba L)^{1/2}}
\exp\left(-\frac{|\eps\theta L +x_1+y_1|^2}{2\eps(1-L)}\right)\exp\left(-\frac{|x'-y'|^2}{2\eps(1+\ba L)}\right)\mu(\rd y)\rd L
\\
=\frac{C}{\eps^{\frac d2}}
\int_F
\frac{\eps\theta L +x_1+y_1}{(1-L)^{3/2}(1+\ba L)^{1/2}}
\exp\left(-\frac{|\eps\theta L +x_1+y_1|^2-|x_1+y_1|^2}{2\eps(1-L)}\right)
\\
\times\exp\Bigg(-\frac 1\eps\Bigg[\underbrace{\frac{|x_1+y_1|^2}{2(1-L)}+\frac{|x'-y'|^2}{2(1+\ba L)}}_{=\bar I_\text{st}(x,y,L)}\Bigg]\Bigg)\mu(\rd y)\rd L
\\
=\frac{C}{\eps^{\frac d2}}e^{-\frac {\hat\lambda}\eps}
\int_F
\frac{\eps\theta L +x_1+y_1}{(1-L)^{3/2}(1+\ba L)^{1/2}}
\exp\left(-\theta L\frac{\eps\theta L +2x_1+2y_1}{2(1-L)}\right)
\\
\times\exp\left(-\frac{\bar I_\text{st}(x,y,L)-{\hat\lambda}}{\eps}\right)\mu(\rd y)\rd L,
\end{multline*}
where $C$ is again a constant possibly varying from line to line but independent of $\eps>0$.
Due to $\bar I_\text{st}\geq I_\text{st} \geq c\geq {\hat\lambda}$ we can simply write $\frac{\bar I_\text{st}(x,y,L)-{\hat\lambda}}{\eps}\geq \bar I_\text{st}(x,y,L)-{\hat\lambda}$ in the last exponential term in the integral.
Exploiting moreover the very rough bounds $1+\ba L\geq 1-L$, $\eps\theta L\leq \eps\theta\leq 1$ for small $\eps$ as well as $\eps\t L+x_1+y_1\geq \eps\t L$ (recall that $x_1,y_1\geq 0$ in $\bD$) we obtain
\begin{multline*}
\int_E\rho^\eps_\text{st}(x,y)\mu(\rd y)
\leq
\frac{C}{ \eps^{\frac d2}}e^{-\frac {\hat\lambda}\eps}
\int_F
\frac{1 +x_1+y_1}{(1-L)^{2}}
\exp\left(-\eps \frac{\theta^2 L^2}{2(1-L)}\right)
\exp\left(-\bar I_\text{st}(x,y,L)+{\hat\lambda}\right)\mu(\rd y)\rd L
\\
=
\frac{Ce^{\hat\lambda}}{ \eps^{\frac d2}}e^{-\frac {\hat\lambda}\eps}
\int_F
\frac{1 +x_1+y_1}{(1-L)^{2}}
\exp\left(-\eps \frac{\theta^2 L^2}{2(1-L)}\right)\\
\times
\exp\Bigg(-\Bigg[\underbrace{\frac{|x_1+y_1|^2}{2(1-L)}+\frac{|x'-y'|^2}{2(1+\ba L)}}_{=\bar I_\text{st}(x,y,L)}\Bigg]\Bigg)\mu(\rd y)\rd L.
\end{multline*}
Since $\frac{1}{1+\ba L}\geq c=\frac 1{1+|A|}>0 $ as soon as $a>0\iff \ba>-1$ we get
\begin{multline*}
\int_{E}\rho^\eps_\text{st}(x,y)\mu(\rd y)
\leq
\frac{C}{\eps^{\frac d2}}e^{-\frac {\hat\lambda}\eps}
\int_F
\frac{1}{(1-L)^{2}}
\exp\left(-\eps \frac{\theta L^2}{2(1-L)}\right)(1+x_1+y_1)
\\
\hspace{1cm}\times
\exp\left(-\frac{|x_1+y_1|^2}{2}-c|x'-y'|^2\right)\mu(\rd y)\rd L
\\
\leq
\frac{C}{\eps^{\frac d2}}e^{-\frac {\hat\lambda}\eps}
\left(\int_0^1
\frac{1}{(1-L)^{2}}
\exp\left(-\eps \frac{\theta L^2}{2(1-L)}\right)\rd L\right)
\\
\times
\left(\int_\bD (1+x_1+y_1)
\exp\left(-\frac{|x_1+y_1|^2}{2}-c|x'-y'|^2\right)\mu(\rd y)\right).
\end{multline*}
The second integral on the whole spatial domain $\bD$ is clearly absolutely convergent due to the exponential decay in $y_1,y'$.
As for the first $\rd L$ integral we estimate
\begin{align*}
L\leq \frac 12& :\hspace{1cm}
\frac{1}{(1-L)^{2}}
\exp\left(-\eps \frac{\theta L^2}{2(1-L)}\right)
\leq
\frac{1}{(1-L)^{2}}
\leq
\frac 14,
\\
L\geq \frac 12 & :\hspace{1cm}
\frac{1}{(1-L)^{2}}
\exp\left(-\eps \frac{\theta L^2}{2(1-L)}\right)
\leq
\frac{1}{\eps^2}\left(\frac{\eps}{1-L}\right)^2
\exp\left(-c\frac{\eps}{1-L}\right),
\end{align*}
where $c=\t/8>0$.
Being $u\mapsto u^2\exp(-c u)$ globally bounded for $u\geq 0$ we conclude that
$$
\int_0^1
\frac{1}{(1-L)^{2}}
\exp\left(-\eps \frac{\theta L^2}{2(1-L)}\right)\rd L \leq C\left (1+\frac 1{\eps ^2}\right)
$$
and therefore
\begin{equation}
 \label{eq:upper_estimate_rho_st}
\int_E\rho^\eps_\text{st}(x,y)\mu(\rd y)
\leq
\frac{C}{\eps^{\frac d2}}\left(1+\frac{1}{\eps^2}\right)e^{-\frac{{\hat\lambda}}{\eps}}
\end{equation}
for some $C\geq 0$ independent of $\eps$.

Gathering \eqref{eq:upper_estimate_rho_int}\eqref{eq:upper_estimate_rho_st} gives
$$
 \rho^\eps_x(E)
 =
 \int_{E}\rho^\eps_\text{int}(x,y)\rd y
 +
 \int_E\rho^\eps_\text{st}(x,y)\mu(\rd y)
\leq
\frac{C}{\eps^{\frac d2}}\left (1+\frac{1}{\eps^2}\right)e^{-\frac{\hat\lambda}\eps}
$$
uniformly in $\eps\to 0$, which in turn yields the LDP upper bound
$$
\limsup\limits_{\eps\to 0}
\,\eps \log \rho^\eps_x(E) \leq -{\hat\lambda}=-\inf\limits_{y\in \bar E}c(x,y).
$$
% % % % % % % % % % % % % % % % %
\textbf{\underline{Step 2: LDP lower bound.}}
Let
$$
{\check\lambda}
\coloneqq
\inf\limits_{y\in E^\circ}c(x,y).
$$
We want to prove that $\rho_x(E)\gtrsim e^{-\frac{\check\lambda}{\eps}}$.
If $E^\circ=\emptyset$ then $\check\lambda=+\infty$ and the statement is vacuous, so we may equally assume that $E^\circ$ is nonempty.
Because $c(x,\cdot)$ is continuous, coercive, and bounded from below, there exists $y^*\in \overline{E^\circ}$ such that
$$
{\check\lambda}
=
c(x,y^*)
=
\min\Big\{I_\text{int}(x,y^*),I_\text{st}(x,y^*)\Big\}.
$$
As usual we distinguish cases depending on which of the two interior/sticky contributions $I_\text{int},I_\text{st}$ realizes the ``least unlikely of the unlikely''.
\begin{enumerate}
\item
Assume first that ${\check\lambda}=I_\text{st}(x,y^*)\leq I_\text{int}(x,y^*)$.
By definition of $I_\text{st}(x,y)=\min\limits_{L\in[0,1]}\bar I_\text{st}(x,y,L)$ there exists $L^*\in[0,1]$ such that ${\check\lambda}=\bar I_\text{st}(x,y^*,L^*)$.
Fix an arbitrarily small $\eta>0$.
By continuity of $\bar I_\text{st}$ in \eqref{eq:def_I2bar} there exists a small neighborhood $F_\eta\subset F\coloneqq E\times [0,1]$ of $(y^*,L^*)$ with $\mu\otimes\rd L(F_\eta)<+\infty$ such that
$$
\bar I_\text{st}(x,y,L)\leq {\check\lambda}+\eta
\qtext{for}
\mu(\rd y)\rd L\text{-a.e. }(y,L)\in F_\eta.
$$
By \eqref{eq:def_rho2_rhobar2}\eqref{eq:def_rho2} we see that
\begin{align*}
\rho^\eps_x(E)
& =
\int_{E}\rho^\eps_\text{int}(x,y)\mu(\rd y) + \int_{E}\int_0^1\bar \rho^\eps_\text{st}(x,y,L)\rd L\mu(\rd y)
\\
& \geq
\int_{F_\eta}\bar \rho^\eps_\text{st}(x,y,L)\rd L\mu(\rd y)
\\
& =\frac{\theta}{(2\pi\eps)^{\frac d2} }\int_{F_\eta}
\frac{\eps\theta L +x_1+y_1}{(1-L)^{3/2}(1+\ba L)^{1/2}}
\\
& \hspace{1cm}
\times
\exp\left(-\frac{|\eps\theta L +x_1+y_1|^2}{2\eps(1-L)}\right)\exp\left(-\frac{|x'-y'|^2}{2\eps(1+\ba L)}\right)\mu(\rd y)\rd L,
\end{align*}
and isolating the leading term $|\eps\theta L +x_1+y_1|^2=
|x_1+y_1|^2 +\eps\theta L \left(\eps\theta L +2x_1+2y_1\right)$ in the first exponential leads next to
\begin{multline*}
\rho^\eps_x(E)
\geq
\frac{c}{\eps^{\frac d2}}\int_{F_\eta}
\frac{\eps\theta L +x_1+y_1}{(1-L)^{3/2}(1+\ba L)^{1/2}}
\exp\left(-\theta L\frac{\eps\theta L +2x_1+2y_1}{2(1-L)}\right)
\\
\times
\exp\Bigg(-\frac 1\eps\Bigg[\underbrace{\frac{|x_1+y_1|^2}{2(1-L)}+\frac{|x'-y'|^2}{2(1+\ba L)}}_{=\bar I_\text{st}(x,y,L)\leq {\check\lambda}+\eta}\Bigg]\Bigg)\mu(\rd y)\rd L
\\
\geq \frac{c}{\eps^{\frac d2}}e^{-\frac{{\check\lambda}+\eta}\eps}\int_{F_\eta}
\frac{\eps\theta L +x_1+y_1}{(1-L)^{3/2}(1+\ba L)^{1/2}}\exp\left(-\theta L\frac{\eps\theta L +2x_1+2y_1}{2(1-L)}\right)\mu(\rd y)\rd L
\end{multline*}
for some $c>0$ independent of $\eps$ (which will keep varying again from line to line below).
Recalling that $x_1,y_1\geq 0$ and decreasing $F_\eta$ if needed we can assume that $0\leq x_1+y_1\leq 2(x_1+y_1^*)$ in $F_\eta$, and with $\frac{1}{1+AL}\geq \frac 1{1+|A|}$, $\frac 1{1-L}\geq 1$ and $\t L\leq \t,\eps\t L\leq 1$ the previous integral can be bounded from below as
\begin{multline}
 \label{eq:LDP_lower_bound_I2min}
\rho^\eps_x(E)
\geq \frac{c}{\eps^{\frac d2}}e^{-\frac{{\check\lambda}+\eta}\eps}\int_{F_\eta}
\frac{\eps\theta L}{(1+|\ba| )^{1/2}}
\exp\left(-\theta \frac{1+ 4 x_1+4y_1^*}{2(1-L)}\right)\mu(\rd y)\rd L
\\
= \frac{c}{\eps^{\frac d2 -1}} e^{-\frac{{\check\lambda}+\eta}\eps} \int_{F_\eta}L\exp\left(-\frac C{1-L}\right)\mu(\rd y)\rd L
% \\
= \frac{c}{\eps^{\frac d2 -1}} e^{-\frac{{\check\lambda}+\eta}\eps}.
\end{multline}
In the last equality we used that $(y,L)\mapsto Le^{-C/(1-L)}>0$ is locally $\mu(\rd y)\rd L$-integrable together with the fact that $F_\eta$ has positive measure, and as usual $c,C>0$ may depend on the various parameters (including $\eta>0$) but not on $\eps$.
 \item
Consider now the opposite case $\check\lambda=I_\text{int}(x,y^*)<I_\text{st}(x,y^*)$.
If $a\leq 1$ there holds $|x_1+y_1|^2\geq |x_1-y_1|^2$ in \eqref{eq:I_st_explicit} and therefore $I_\text{st}\geq I_\text{int}$, so we may as well assume that $a> 1$.
But, for $a>1$, from the proof of Lemma~\ref{lem:min_I1_I2} we have $I_\text{int}=I_\text{st}$ if either $x_1=0$ or $y^*_1=0$, see Figure~\ref{fig:cones_I2}.
Hence we only consider $a>1$ and $x_1,y_1^*>0$.

This being said, fix again a small $\eta>0$ and choose a small neighborhood $E_\eta\subset E$ of $y^*$ with positive Lebesgue measure (recall that $E^\circ\neq\emptyset$) and such that $I_\text{int}(x,y)\leq {\check\lambda} +\eta$, $\mu$-a.e. on $E_\eta$.
Up to decreasing $E_\eta$ if needed we can further assume that $x_1y_1\geq \frac 12x_1y_1^*>0$ in this neighborhood.
By \eqref{eq:def_rho1} and $\mu(\rd y)\geq \rd y$, elementary algebra in the exponentials gives
\begin{multline}
\label{eq:LDP_lower_bound_I1min}
\rho^\eps_x(E)
=\int_{E}\rho^\eps_\text{int}(x,y)\mu(\rd y) + \int_E \rho^\eps_\text{st}(x,y)\mu(\rd y)
\geq
\int_{E_\eta}\rho^\eps_\text{int}(x,y)\rd y
\\
=\int_{E_\eta}\frac{1}{(2\pi \eps)^{\frac d2}}
\left[\exp\left(-\frac{|x_1-y_1|^2}{2\eps}\right)-\exp\left(-\frac{|x_1+y_1|^2}{2\eps}\right)\right]\exp\left(-\frac{|x'-y'|^2}{2\eps}\right)
\rd y
\\
=
\frac{c}{\eps^{\frac d2}}\int_{E_\eta}
\Bigg[1-\exp\Big(-\underbrace{\frac{2x_1y_1}{\eps}}_{\geq\frac{x_1y_1^*}{\eps}}\Big)\Bigg]
\exp\Bigg(-\frac 1\eps\Bigg[\underbrace{\frac{|x_1-y_1|^2}{2}+\frac{|x'-y'|^2}{2}}_{I_\text{int}(x,y)\leq {\check\lambda}+\eta}\Bigg]\Bigg)
\rd y
\\
\geq
\frac{c}{\eps^{\frac d2}}e^{-\frac{{\check\lambda}+\eta}{\eps}}\int_{E_\eta}\left[1-\exp\left(-\frac{x_1y^*_1}{\eps}\right)\right] \rd y
\\
\geq \frac{c}{\eps^\frac d2}e^{-\frac{{\check\lambda}+\eta}{\eps}}\int_{E_\eta}\frac 12\rd y
= \frac{c}{\eps}e^{-\frac{{\check\lambda}+\eta}{\eps}},
\end{multline}
where $c>0$ is again uniform in $\eps$.

\end{enumerate}
Gathering \eqref{eq:LDP_lower_bound_I2min}\eqref{eq:LDP_lower_bound_I1min} and taking $\eps\to 0$ we get in any case
$$
\liminf\limits_{\eps\to 0}\eps\log  \rho^\eps_x(E)\geq -({\check\lambda}+\eta).
$$
Since $\eta>0$ was arbitrary this gives the desired LDP lower bound the proof is complete.
\end{proof}

%%%%%%%%%%%%%%%%%%%%%%%%%%%%%%%%%%%%%%%%%%%%%%%%%%%%%%%%%%%%%%%%%%%%%%%%%%%%%%%%%%%%
\section{Study of the cost function}
\label{sec:cost}
In order to support our informal discussion from Section~\ref{sec:heuristics} we need now to rigorously identify the static cost from Theorem~\ref{theo:LDP_static} as a dynamical minimization problem $c(x,y)=\min\int_0^1L(\o_t,\dot\o_t)\rd t$, where the Lagrangian $L(x,q)$ is obtained by Legendre-transforming the Hamiltonian
$$
H\phi(x)=H_a(x,\nabla\phi(x))=\lim\limits_{\eps\to 0}\eps e^{-\frac\phi\eps}\Q^\eps e^{\frac\phi\eps}.
$$
Again, we only carry out this program in the easy planar setup for simplicity (our explicit computation of geodesics below becomes impossible in general manifolds).
Claiming now full mathematical rigour, for fixed $a>0$ and all $p\in T_x^*\bar D,q\in T_x\bar D$ we set according to \eqref{eq:Heps}\eqref{eq:H}\eqref{eq:L}
$$
H(x,p)\coloneqq
\frac 12
\begin{cases}
|p|^2 & \text{if }x\in D\\
a|p|^2& \text{if }x\in \pD
\end{cases}
$$
and
\begin{equation}
\label{eq:def_La}
L(x,q)\coloneqq H^*(x,q)=
\frac 12
\begin{cases}
|q|^2 & \text{if }x\in D\\
\frac 1a |q|^2& \text{if }x\in \pD
\end{cases}.
\end{equation}
Implicitly we mean here that $q\in T_x\pD\cong \R^{d-1}$ should be a vertical tangent vectors if $x\in\pD$ while $q\in T_x D\cong\R^d$ at interior points $x\in D$.
According to our discussion in Section~\ref{sec:heuristics} $L$ is only lower semi-continuous for $a>1$, and we expect that the lower semi-continuous relaxation is the relevant object.
Here it is a simple exercise to compute
\begin{equation}
\label{eq:def_Lbar}
 \uL(x,q)
\coloneqq\underset{\substack{x_n\to x\\q_n\to q}}{\inf}\underset{n\to\infty}{\liminf} L(x_n,q_n)
=
\begin{cases}
L(x,q) & \text{if }a>1\\
\frac 12 |q|^2 & \text{if }a\leq 1\\
\end{cases}.
\end{equation}
The associated path-action is
\begin{equation}
\label{eq:def_C_dyn}
\underline C(\o)
\coloneqq
\begin{cases}
\int_0^1\uL(\o_t,\dot\o_t)\rd t & \text{if }\o\in H^1,\\
+\infty & \text{else}.
\end{cases}
\end{equation}
Recalling that the contact angle $\alpha$ is defined as
$$
\sin^2\alpha=\frac 1a
\hspace{2cm}\text{for }a\geq 1,
$$
our main result in this section is
\begin{theo}
\label{theo:c_static_dynamic-geodesics}
The dynamical cost $\uC$ is $H^1$-weakly lower semi-continuous, coercive, and induces the static cost $c$ from Theorem~\ref{theo:LDP_static} in the sense that
\begin{equation}
\label{eq:c=min_L}
  \boxed{c(x,y)
  = \min\left\{\int_0^1 \uL(\o_t,\dot \o_t)\rd t:\quad \o\in H^1(0,1;\bar D)\mbox{ and }\o_0=x,\o_1=y\right\}.}
\end{equation}
The minimum is always attained for a unique geodesic $(\gamma^{xy}_t)_{t\in[0,1]}\in H^1(0,1;\bar D)$.
Moreover
\begin{itemize}
 \item
For $a\leq 1$ geodesics are the standard Euclidean interpolants $\gamma^{xy}_t=(1-t)x+ty$.
\item
For $a>1$ geodesics are characterized as in Figure~\ref{fig:cost_geodesics_x1>0} and Figure~\ref{fig:cost_geodesics_x1=0}:
for $y\not\in\cone(x)$ they are concatenations of at most three straight lines, either making an angle $\alpha$ with the normal direction or vertical along the boundary, while for $y\in \cone(x)$ geodesics are again Euclidean interpolants.
\end{itemize}
\end{theo}
\noindent
Not only this is a dynamical characterization of our intrinsic distance, but it will also be useful in our next section~\ref{sec:dynamical_LDP} to derive the sample-path LDP from the static one in Theorem~\ref{theo:LDP_static}.
\\

Since $\uL(x,\cdot)$ is quadratic in its second argument, an immediate consequence is that our static cost, as a short-time rate function, truly determines the intrinsic distance of SBM:
\begin{cor}
 $c(x,y)=d^2(x,y)$ is a squared distance.
\end{cor}
\begin{proof}
This is very classical so we only sketch the proof.
Assume that $c(x,y)=0$.
According to \eqref{eq:def_La}\eqref{eq:def_Lbar} there holds $\uL(x,q)\geq C|q|^2$ for $C=\min\{1/2,1/2a\}>0$.
In that case the unique minimizing geodesic must satisfy $0=c(x,y)=\int_0^1\uL(\g^{xy}_t,\dot\g^{xy}_t)\rd t\geq C\int_0^1|\dot \g^{xy}_t|^2\rd t$, hence $x=\g^{xy}_0=\g^{xy}_1=y$.
The symmetry is obvious in view of \eqref{eq:c=min_L} and $\uL(x,q)=\uL(x,-q)$.
For the triangular inequality, pick any $x,y,z$.
For given $\lambda\in(0,1)$ we can construct an admissible path $(\o_t)_{t\in[0,1]}$ starting from $\o_0=x$, passing through $\o_\lambda=y$ and finally $\o_1=z$ by following first the geodesic $\g^{xy}$ rescaled in time $t\in[0,\lambda]$ and then the geodesic $\g^{yz}$ in time $t\in[\lambda,1]$.
By quadratic scaling this gives a cost $d^2(x,z)\leq \int_0^1\uL(\o_t,\dot\o_t)=\frac 1\lambda d^2(x,y)+\frac{1}{1-\lambda}d^2(y,z)$, and choosing $\lambda=\frac{d(x,y)+d(y,z)}{d(x,y)}$ gives the triangular inequality $d^2(x,z)=[d(x,y)+d(y,z)]^2$.
\end{proof}

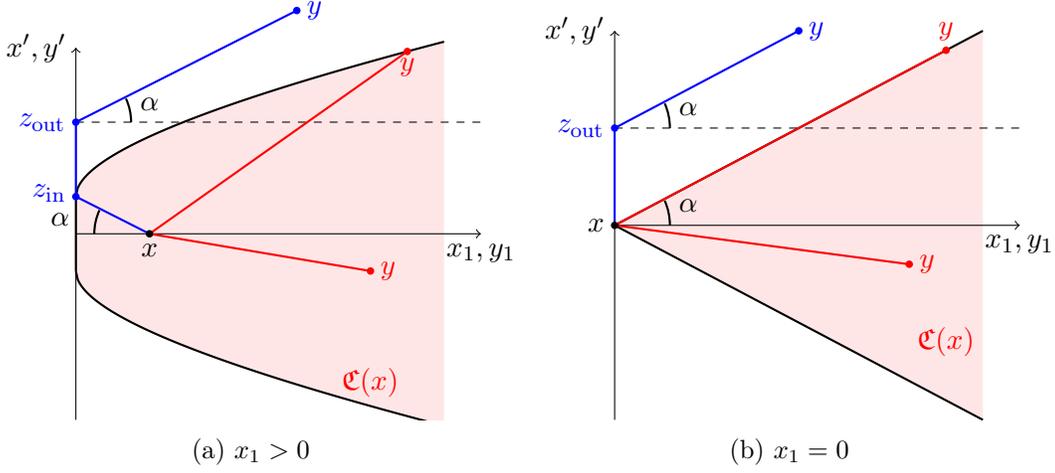
\begin{figure}[h!]
\begin{center}
\begin{subfigure}{0.4\textwidth}
\begin{tikzpicture}[]
\begin{axis}[ticks=none,axis lines=none,
xmin=-2, xmax=12,
ymin=-5, ymax=6.5
]

\addplot [domain=0:10, samples=200, name path=f, thick]
{1+.1*x+sqrt(x)};
\addplot [domain=0:10, samples=200, name path=g, thick]
{-1-.1*x-sqrt(x)};
\draw [thick] (axis cs:0,-1) -- (axis cs:0,1);
\addplot[red!10, opacity=0.4] fill between [of= f and g];
% \draw[help lines,step=1] (0,-5) grid (10,5);
\draw [->] (axis cs:0,0) -- (axis cs:11,0);
\draw [->] (axis cs:0,-5) -- (axis cs:0,5);
\draw [blue, thick] (axis cs:2,0) -- (axis cs:0,1) -- (axis cs:0,3) -- (axis cs:6,6);
\draw [dashed] (axis cs:0,3) -- (axis cs:11,3);

\draw [thick] (axis cs:1.5,3) arc (0:25:1.5);
\node at (axis cs:2,3.5)[]{$\alpha$};
\draw [thick] (axis cs:0.5,0) arc (180:155:1.5);
\node at (axis cs:.1,.4)[left]{$\alpha$};

\node at (axis cs:0,1)[blue,circle,fill,inner sep=1pt]{};
\node at (axis cs:0,1.1)[left,blue]{$z_\text{in}$};
\node at (axis cs:0,3)[blue,circle,fill,inner sep=1pt]{};
\node at (axis cs:0,3)[left,blue]{$z_\text{out}$};
\node [below] at (axis cs:11,0) {$x_1,y_1$};
\node [left] at (axis cs:0,5) {$x',y'$};

\node at (axis cs:6,6)[blue,circle,fill,inner sep=1pt]{};
\node [right,blue] at (axis cs:6,6) {$y$};

\draw [red, thick] (axis cs:2,0) -- (axis cs:9,4.9);
\node at (axis cs:9,4.9)[red,circle,fill,inner sep=1pt]{};
\node [below,red] at (axis cs:9,5) {$y$};
\node [red] at (axis cs:8,-4) {$\cone(x)$};

\draw [red, thick] (axis cs:2,0)  -- (axis cs:8,-1);
\node at (axis cs:8,-1)[circle,fill,inner sep=1pt,red]{};
\node [right,red] at (axis cs:8,-1) {$y$};
% \node at (axis cs:4.5,-4.5) {$c=c_{Sti}$};
% \node at (axis cs:4,0) {$x^2$};
\node at (axis cs:2,0)[circle,fill,inner sep=1pt]{};
\node [below] at (axis cs:2,0) {$x$};
\end{axis}
  \end{tikzpicture}
  \caption{$x_1>0$}
\label{fig:cost_geodesics_x1>0}
  \end{subfigure}
\hspace{.3cm}
\begin{subfigure}{0.4\textwidth}
\begin{tikzpicture}[]
\begin{axis}[ticks=none,axis lines=none,
xmin=-2, xmax=12,
ymin=-5, ymax=6
]

\addplot [domain=0:10, samples=200, name path=f, thick]
{.5*x};
\addplot [domain=0:10, samples=200, name path=g, thick]
{-.5*x};
\addplot[red!10, opacity=0.4] fill between [of= f and g];
\draw [->] (axis cs:0,0) -- (axis cs:11,0);
\draw [->] (axis cs:0,-5) -- (axis cs:0,5);

\draw [blue, thick] (axis cs:0,0)  -- (axis cs:0,2.5) -- (axis cs:5,5);
\draw [dashed] (axis cs:0,2.5) -- (axis cs:11,2.5);
\draw [thick] (axis cs:1.5,2.5) arc (0:25:1.5);
\node at (axis cs:2,3)[]{$\alpha$};

\node at (axis cs:0,2.5)[blue,circle,fill,inner sep=1pt]{};
\node at (axis cs:0,2.5)[blue,left]{$z_\text{out}$};
\node at (axis cs:5,5)[blue,circle,fill,inner sep=1pt]{};
\node [right,blue] at (axis cs:5,5) {$y$};
\draw [thick] (axis cs:1.5,0) arc (0:25:1.5);
\node at (axis cs:2,.5)[]{$\alpha$};

\draw [red, thick] (axis cs:0,0)  -- (axis cs:9,4.5);
\node at (axis cs:9,4.5)[circle,fill,inner sep=1pt,red]{};
\node [above,red] at (axis cs:9,4.5) {$y$};

\draw [red, thick] (axis cs:0,0)  -- (axis cs:8,-1);
\node at (axis cs:8,-1)[circle,fill,inner sep=1pt,red]{};
\node [right,red] at (axis cs:8,-1) {$y$};

% \node at (axis cs:0,1)[red,circle,fill,inner sep=1pt]{};

\node [below] at (axis cs:11,0) {$x_1,y_1$};
\node [left] at (axis cs:0,5) {$x',y'$};
\node at (axis cs:0,0)[circle,fill,inner sep=1pt]{};
\node [left] at (axis cs:0,0) {$x$};

\node [red] at (axis cs:9,-3) {$\cone(x)$};
% \node at (axis cs:4,0) {$x^2$};
\end{axis}
  \end{tikzpicture}
  \caption{$x_1=0$}
\label{fig:cost_geodesics_x1=0}
\end{subfigure}
\end{center}
\caption{Geodesics for all possible configurations of $(x,y)\in \bD^2$}
\label{fig:fig:cost_geodesics}
\end{figure}

\begin{proof}[Proof of Theorem~\ref{theo:c_static_dynamic-geodesics}]
For $a\leq 1$, $\uL(x,q)=\frac 12|q|^2$ is the standard Euclidean cost and the whole statement is obvious so we only address the case $a>1$.

Let us first prove that $C$ is a well-behaved functional.
For $a>1$ we have $\uL=L$, hence by \eqref{eq:def_La}
\begin{equation}
\label{eq:L_coercive}
\uL(x,q)\geq \frac{1}{2a}|q|^2,\hspace{1cm}
\forall\,x\in \bar D,\,q\in T_x\bar D
\end{equation}
(this is of course a very bad estimate if $x\in D$, when $\uL(x,q)=\frac 12|v|^2$).
This immediately gives $\dot H^1$ coercivity of
$$
\omega\mapsto C(\o)\coloneqq \int _0^1 \uL(\o_t,\dot\o_t)\rd t
\geq \frac{1}{2a}\int_0^1|\dot\o_t|^2\rd t.
$$
For the weak lower semi-continuity, consider any fixed endpoints $\o_0=x,\o_1=y$ and let $\omega_n\rightharpoonup \o$ weakly in $H^1$.
By usual compactness $H^1\subset\subset C([0,1])$ we can assume that $\o_n\to \o$ uniformly.
Let
$$
Z\coloneqq \Big\{t\in (0,1):\quad \o_t\in \pD\Big\}
\qqtext{and}
J\coloneqq (0,1)\setminus Z=\Big\{t\in (0,1):\quad \o_t\in D\Big\}
$$
By standard properties of Sobolev functions we see that $\dot \o_t$ is tangent to $\pD$ for a.e. $t\in Z$.
By standard weak lower semi-continuity we get
$$
\int_Z \uL(\o_t,\dot\o_t)
=
\int_Z\frac{1}{2a}|\dot\o_t|^2
\leq
\liminf\limits_{n\to\infty}\int_Z\frac{1}{2a}|\dot\o_t^n|^2
% \\
\leq
\liminf\limits_{n\to\infty}\int_Z \uL(\o^n_t,\dot\o^n_t)\rd t,
$$
where the last very rough inequality is just \eqref{eq:L_coercive}.
In order to get control over $J$, fix an arbitrarily small $\eta>0$ and let
$$
J_\eta
\coloneqq \Big\{t\in [0,1]:\quad \operatorname{dist}(\o_t,\pD)\geq \eta\Big\}\subseteq J.
$$
Since $\o\in H^1$ is uniformly continuous we have that $|J\setminus J_\eta|\to 0$ as $\eta\to 0$.
Crucially, for fixed $\eta$ the uniform convergence $\o^n\to \o$ guarantees that $\operatorname{dist}(\o^n_t,\pD)\geq \eta/2>0$ for sufficiently large $n$.
As a consequence
\begin{multline*}
\int_{J_\eta} \uL(\o_t,\dot \o_t)\rd t
=
\int_{J_\eta} \frac 12|\dot \o_t|^2\rd t
\leq
\liminf\limits_{n\to\infty}
\int_{J_\eta} \frac 12|\dot \o_t^n|^2\rd t
\\
=
\liminf\limits_{n\to\infty}
\int_{J_\eta} \uL(\o^n_t,\dot \o_t^n)\rd t
\leq
\liminf\limits_{n\to\infty}
\int_{J} \uL(\o^n_t,\dot \o_t^n)\rd t,
\end{multline*}
where the last inequality holds simply because $J_\eta\subseteq J$.
Since $\eta>0$ was arbitrary and $|J\setminus J_\eta|\to 0$ as $\eta\to 0$ we conclude that
$$
\int_{J} \uL(\o_t,\dot \o_t)\rd t
=\lim\limits_{\eta\to 0}\int_{J_\eta} \uL(\o_t,\dot \o_t)\rd t
\leq
\liminf\limits_{n\to\infty}
\int_{J} \uL(\o^n_t,\dot \o_t^n)\rd t
$$
and therefore
\begin{multline*}
\int_0^1 \uL(\o_t,\dot \o_t)\rd t
=
\int_Z \uL(\o_t,\dot \o_t)\rd t + \int_J \uL(\o_t,\dot \o_t)\rd t
\\
\leq
\liminf_{n\to\infty}\int_Z \uL(\o^n_t,\dot \o^n_t)\rd t +
\liminf_{n\to\infty}\int_J \uL(\o^n_t,\dot \o^n_t)\rd t
\leq
\liminf_{n\to\infty}\int_0^1 \uL(\o^n_t,\dot \o^n_t)\rd t.
\end{multline*}
This shows that the Lagrangian cost $C$ is lower-semicontinuous and coercive, and we can now focus on the geodesic problem itself.
In other words, for fixed $x,y\in\bD$ let us solve
$$
\min\left\{\int_0^1 \uL(\o_t,\dot \o_t)\rd t:\quad \o\in H^1(0,1;\bar D)\mbox{ with }\o_0=x,\o_1=y\right\}.
$$
First, note that linear interpolation $\o_t=(1-t)x+t y$ is always an admissible competitor with finite cost, hence there exists a minimizing sequence $\{\o^n\}_{n\geq 0}\subset H^1$.
Since the endpoints are fixed, the coercivity \eqref{eq:L_coercive} gives $\|\o^n\|_{H^1}\leq C$ and there exists a weak-$H^1$ limit $\omega$.
By lower-semicontinuity $\o$ is necessarily a minimizer and it suffices to show that it has the right cost $c(x,y)=\int_0^1 L(\o_t,\dot\o_t)\rd t$.
We distinguish cases depending on whether $x,y$ belong to $\pD$ or not.

\begin{enumerate}
 \item
If $x_1=y_1=0$, recall that $\uL(\o_t,\dot\o_t)\geq \frac{1}{2a}|\dot\o_t|^2$ with equality if and only if $\o_t\in \pD$.
Clearly in that case the best one can do is follow the Euclidean geodesic $\gamma^{xy}_t=(1-t)x+ty$ along the vertical boundary, for a cost $c(x,y)=\frac{1}{2a}|y-x|^2=\int_0^1\frac{1}{2a}|\dot \gamma^{xy}_t|^2\rd t$ as in \eqref{eq:c_mu>0}.

\item
Consider now $x_1>0$ and $y_1>0$ as in Figure~\ref{fig:cost_geodesics_x1>0}
\begin{itemize}
 \item
Pick first $y\not\in\cone(x)$, and take $z_\text{in},z_\text{out}\in \pD$ as in Figure~\ref{fig:cost_geodesics_x1>0}, implicitly defined through the condition of making an angle $\alpha$ with the normal (elementary geometry shows this is possible if $y\not\in\cone(x)$).
The blue curve $x\leadsto z_\text{in}\leadsto z_\text{out}\leadsto y$ is always admissible.
Suitably adjusting the constant speed on each segment so that $\uL(\o_t,\dot\o_t)=cst$ for $t\in [0,1]$, the definition of $\sin^2\alpha=\frac 1a$ precisely results in the overall cost $\frac{1}{2a}\left(\sqrt{\ba} (x_1+y_1) + |y'-x'|\right)^2$ as in \eqref{eq:c_mu>0} (we omit the lengthy but straightforward computation here).
Let us check that this is optimal.
Indeed, since $y\in\cone(x) \Leftrightarrow (x,y)\in \cone$, and by definition of $\cone$ this particular cost is strictly less than the Euclidean cost $\frac 12|y-x|^2$, which is clearly the best one could do if the geodesic remained contained in $D$ and away from $\pD$ for all times $t\in[0,1]$ (since then the Lagrangian would just be $\int_0^1\frac 12|\dot\o_t|^2\rd t$).
This means that, for $y\not\in \cone(x)$, a geodesic must eventually go through the boundary at some point.
By continuity, there are a first entry and last exit times $0<t_\text{in}\leq t_\text{out}<1$ such that $Z_\text{in}\coloneqq\o_{t_\text{in}}$ and $Z_\text{out}\coloneqq\o_{t_\text{out}}$ lie on $\pD$.
Clearly in that case the best one can do is to first follow a geodesic $x+\frac{t}{t_\text{in}}(Z_\text{in}-x)$, then a vertical segment $Z_\text{in}+\frac{t-t_\text{in}}{t_\text{out}-t_\text{in}}(Z_\text{out}-Z_\text{in})$, and finally connect to $y$ via $Z_\text{out}+\frac{t-t_\text{out}}{1-t_\text{out}}(y-Z_\text{out})$.
This has overall cost $\int_0^{t_\text{in}}\left|\frac{Z_\text{in}-x}{t_\text{in}}\right|^2\rd t +\int_{t_\text{in}}^{t_\text{out}}\frac 1{2a}\left|\frac{Z_\text{out}-Z_\text{in}}{t_\text{out}-t_\text{in}}\right|^2\rd t +\int_{t_\text{out}}^1\left|\frac{y-Z_\text{out}}{1-t_\text{out}}\right|^2\rd t$.
(The $1/2a$ factor along the vertical path is crucial here!)
Optimizing first in $t_\text{in},_\text{out}$ for fixed $Z_\text{in},Z_\text{out}$, and then minimizing with respect to $Z_\text{in},Z_\text{out}\in \pD$, a tedious but straightforward computation shows that the optimal cost in this case is realized when $Z_\text{in}=z_\text{in}$ and $Z_\text{out}=z_\text{out}$ are precisely defined through the $\alpha$-angle condition as in Figure~\ref{fig:cost_geodesics_x1>0} (we omit again the details).
\item
If now $y\in\cone(x)$ we claim that the Euclidean geodesic with cost $\frac 12|y-x|^2$ is optimal, see again Figure~\ref{fig:cost_geodesics_x1>0}.
Clearly this is the best one can achieve without ever going through $\pD$, so it is enough to prove that any path actually going through $\pD$ has larger cost.
For any such path $\o$, let as before $0<t_\text{in}\leq t_\text{out}<1$ be the first entry and last exit times, with $Z_\text{in}=\o_{t_\text{in}}\in \pD$ and $Z_\text{out}=\o{t_\text{in}}\in\pD$.
Exactly as before, the best strategy is then given explicitly by three consecutive straight lines $x\leadsto Z_\text{in}\leadsto Z_\text{out}\leadsto y$, each with constant speed suitably determined by $t_\text{in}\leq t_\text{out}$.
Again, this can be explicitly optimized first w.r.t to $t_\text{in}\leq t_\text{out}$ for given $Z_\text{in},Z_\text{out}$, and then with respect to the locations of $Z_\text{in},Z_\text{out}$.
The main difference is that the condition that $y\in \cone(x)$ leads now to $Z_\text{in}=Z_\text{out}$, uniquely determined by normal reflection on the boundary (this is geometrically intuitive and clearly gives the shortest path from $x\in D$ to $y\in D$ if forced to go through the boundary at least once, so we skip the details).
The resulting cost turns out to be exactly $\frac 12(|x-_\text{in}|+|Z_\text{out}-y|)^2>\frac 12|x-y|^2$.
This cannot be optimal, since the cost $\frac 12|x-y|^2$ is realized in particular by the Euclidean interpolation.
As a consequence a minimizing curve cannot cross the boundary at any time, and the Euclidean geodesic $\gamma^{xy}_t=(1-t)x+t y$ is therefore the sought geodesic.
\end{itemize}
\item
Finally we consider $x_1=0$ and $y_1>0$ (we already covered $x_1=y_1=0$ and $x_1>0,y_1>0$, so all the remaining cases are covered by simple symmetry $c(x,y)=c(y,x)$).
\begin{itemize}
 \item
Pick first $y\not\in\cone(x)$.
Similarly to a previous case, it is easy to check by hand that if $z_\text{out}$ is determined through the $\alpha$-angle condition as in Figure~\ref{fig:cost_geodesics_x1=0} and the constant speed is suitably adjusted on each segment of the blue curve, the cost is exactly $\frac{1}{2a}\left(\sqrt{\ba} (x_1+y_1) + |y'-x'|\right)^2$.
By \eqref{eq:c_mu>0} this realizes the cost $c(x,y)$.
We claim that this is necessarily optimal in the Lagrangian minimization.
Indeed, by definition of $\cone(x)$ this cost is strictly less than the Euclidean $\frac 12|x-y|^2$, hence a minimizing geodesic must necessarily go through the boundary at some point, until a last exit time $t_\text{out}\in [0,1)$ with $Z_\text{out}\in \pD$.
The case $t_\text{out}=0$ would lead to the suboptimal Euclidean cost and is therefore ruled out.
By definition of $\bar L$ clearly the best one can do is then to move vertically $x\leadsto Z_\text{out}$ along $\pD$ for $t\in [0,t_\text{out}]$ (wandering off of $\pD$ would result in unnecessary oscillations in the horizontal direction as well as more costly vertical motion due to $a>1$), and then remain contained in $D$ for later times $t\in(t_\text{out},1]$.
In this scenario clearly the best strategy is to follow two consecutive straight lines $x\leadsto Z_\text{out}\leadsto y$, and optimizing first w.r.t. $t_\text{out}$ for fixed $Z_\text{out}$, and then w.r.t $Z_\text{out}$, one gets that $Z_\text{out}=z_\text{out}$ is determined exactly by the $\alpha$-angle condition with the right cost.
This means that the blue curve in Figure~\ref{fig:cost_geodesics_x1=0} is optimal as claimed.
\item
Now if $y\in \cone(x)$ the same optimization program easily leads to $Z_\text{out}=x$ and $t_\text{out}=0$, i-e the geodesic immediately exits $\pD$ and never reenters.
But in that case $\o_t\in D$ for all $t\in(0,1]$, the Lagrangian is simply $\int_0^1\frac 12|\dot \omega_t|^2\rd t$, so the minimal cost is realized by Euclidean interpolation as in our statement and the proof is complete.
\end{itemize}
\end{enumerate}
\end{proof}

%%%%%%%%%%%%%%%%%%%%%%%%%%%%%%%%%%%%%%%%%%%%%%%%%%%%%%%%%%%%%%%%%%%%%%%%%%%%%%%%%%%%
\section{Sample path Large Deviation Principle}
\label{sec:dynamical_LDP}
Recall that $\O=C([0,1];\bD)$ and ${R^\eps_x}\in \P(\O)$ is the path measure for SBM started from $x\in \bar D$, slowed-down on the scale $\eps>0$ as before.
The goal of this section is to establish the dynamical counterpart of Theorem~\ref{theo:LDP_static}
\begin{theo}
\label{theo:LDP_dynamic}
For fixed $x\in\bar D$ the sequence $\left\{{R^\eps_x}\right\}_{\eps>0}\in \P(\O)$ satisfies a Large Deviation Principle
$$
\boxed{R^\eps_x\underset{\eps\to 0}{\asymp}\exp\left(-\frac 1\eps \uC_x(\o)\right)}
$$
for the uniform topology on $\Omega$, with good rate function
\begin{equation*}
\uC_x(\o)=\iota_{\{\o_0=x\}} + \uC(\o)
\end{equation*}
and $\uC$ as in \eqref{eq:def_C_dyn}.
\end{theo}

This is nothing but a Schilder's theorem for sticky-reflected Brownian motion.
Our interest for such dynamical statement is twofold.
First, it provides information at the path level, which conveys more information than the static LDP in Theorem~\ref{theo:LDP_static}.
Second, this allows a direct application abstract results guaranteeing Gamma-convergence of the $\eps$-Schr\"odinger problem to the deterministic counterpart, see e.g. \cite[prop. 2.5 and thm. 2.7]{leonard2012schrodinger}.

Note that our static result from Theorem~\ref{theo:LDP_static} asserts nothing but the LDP for the initial-terminal joint distribution for $\delta_x\otimes \rho^\eps_x=(e_0,e_1)\pf {R^\eps_x}$.
Just as in one of the classical proofs of Schilder's theorem for pure Brownian motion, the main and standard idea to establish Theorem~\ref{theo:LDP_dynamic} will be to perform a time slicing, retrieve a discrete LDP via the Chapman-Kolmogorov property and iterated applications of our previous static LDP, and finally apply the Dawson-G\"artner theorem to recover the dynamic LDP by projective limit.
More precisely, we will first establish
\begin{prop}
 \label{prop:LDP_dynamic_pointwise}
 Theorem~\ref{theo:LDP_dynamic} holds for the topology of pointwise convergence on $\O$.
\end{prop}
To this end, fix a large $N\in \N$ and take any partition $\tau_N$ of the time interval $[0,1]$
$$
0=t_0<t_1<\dots<t_N=1
$$
with size $|\tau_N|=\max_i|t_{i+1}-t_i|\to 0$ as $N\to\infty$.
For a fixed curve $\o\in \O$ with $\o_0=x$ we denote
$$
y_j\coloneqq \o_{t_j}, \hspace{1cm}j=0,\dots ,N,
$$
and
$$
R^{\eps,N}_x:=(e_{t_0},\dots,e_{t_N})\pf {R^\eps_x}
\hspace{1cm}\in \P(\bar D^{N+1}).
$$
In other words,
$$
R^{\eps,N}_x(E_0\times\dots\times E_N)={R^\eps_x}(\o_{t_0}\in E_0,\dots ,\omega_{t_N}\in E_N)
=
\PP_x(X^\eps_{t_0}\in E_0,\dots,X^\eps_{t_N}\in E_N)
$$
for the slowed-down sticky Brownian motion $(X^{\eps}_t)_{t\in[0,1]}$.
\begin{lem}
\label{lem:LDP_slicing}
For any fixed $x\in \bD$ and $N<+\infty$ the sequence $\left\{R^{\eps,N}_x\right\}_{\eps>0}$ satisfies the LDP
$$
R^{\eps,N}_x
\underset{\eps\to 0}{\asymp}
\exp\left(-\frac 1\eps C_x^N(y_0,\dots,y_N)\right)
$$
with good rate function
\begin{equation}
\label{eq:def_discrete_Lagrangian_LN}
C^N_x(y_0,\dots,y_N)
\coloneqq
\iota_{\{y_0=x\}} + C^N(y_0,\dots,y_N)
\qtext{with}
 C^N(y_0,\dots,y_N)
 \coloneqq
 \sum\limits_{j=0}^{N-1}\frac{c(y_j,y_{j+1})}{t_{j+1}-t_j},
\end{equation}
where the static cost $c(x,y)$ is exactly as in \eqref{eq:c_mu<0}\eqref{eq:c_mu>0}.
\end{lem}
\begin{proof}
The argument is very similar to the proof of Theorem~\ref{theo:LDP_static} so we only sketch the details.
Pick any Borel sets $E_0,\dots,E_N\subset \bD$, and let us denote for convenience
$$
\eps_j:=\eps(t_{j+1}-t_{j})
\qqtext{and}
E\coloneqq E_1\times\dots\times E_N.
$$
By the Chapman-Kolmogorov property we have that
\begin{multline*}
 R^{\eps,N}_x(E_0\times\dots\times E_N)
 =\1_{\{x\in E_0\}}\int_{E} p_{\eps_0}(x,\rd y_1)p_{\eps_1}(y_1,\rd y_2)\dots p_{\eps_{N-1}}(y_{N-1},\rd y_N)
 \\
 =\1_{\{x\in E_0\}}
 \int_{E} \rho^{\eps_0}_x(\rd y_1)\rho^{\eps_1}_{y_1}(\rd y_2)\dots \rho^{\eps_{N-1}}_{y_{N-1}}(\rd y_N).
\end{multline*}
Just as in \eqref{eq:def_Reps_2d}\eqref{eq:def_rho1}\eqref{eq:def_rho2_rhobar2}\eqref{eq:def_rho2} the kernels $\rho^{\eps_j}_{y_j}(\rd y_{j+1})$ are explicitly given by
$$
\rho^{\eps_j}_{y_j}(\rd y_{j+1})
=
\rho^{\eps_j}_\text{int}(y_j,y_{j+1})\rd y_{j+1} + \left(\int_0^1\bar \rho^{\eps_j}_\text{st}(y_j,y_{j+1},L_j)\rd L_j\right)\mu(\rd y_{j+1}).
$$
For finite $N$, it is a tedious but rather straightforward exercise to adapt $N$ consecutive times the same argument as in the proof of Theorem~\ref{theo:LDP_static} to get the LDP for $R^{\eps,N}_x$ as $\eps\to 0$.
Let us just sketch the idea.
For the LDP lower bound, pick first a minimizer $y^*=(y^*_0,\dots,y^*_N)$ of $C^N_x$ in $\bar E=\bar E_1\times\dots\times\bar E_N$.
In \eqref{eq:def_discrete_Lagrangian_LN} clearly one has two possible rate functions $I_\text{int}(y_j,y_{j+1})$ and $I_\text{st}(y_j,y_{j+1})=\min\limits_{L_j\in[0,1]}\bar I_\text{st}(y_j,y_{j+1},L_j)$ competing for $c(y_j,y_{j+1})=\min\left\{I_\text{int}(y_j,y_{j+1}),I_\text{st}(y_j,y_{j+1})\right\}$.
Distinguishing which one is acting separately on each subinterval $[t_j,t_{j+1}]$, each with scale $\eps_j=\eps (t_{j+1}-t_j)$, one simply guesses as in the proof of Theorem~\ref{theo:LDP_static} how to pick either an $\eta$-neighborhood $E_{j,\eta}\subset E_j$ of $y^*_{j+1}$ or a neighborhood $F_{j,\eta}\subset E_j\times[0,1]$ of $(y^*_{j+1},L^*_j)$ (depending on whether $I_\text{int}$ or $I_\text{st}$ is smaller) in order to get the LDP lower bound, up to an arbitrarily small correction $\eta$ (see again the details of the proof of Theorem~\ref{theo:LDP_static}).
The LDP upper bound is even simpler, given that the upper estimates used in the static proof were actually global in $x,y,L$ (at least to a sufficient extent so that they can be iterated $N$ times).
The $t_{j+1}-t_j$ scaling in \eqref{eq:def_discrete_Lagrangian_LN} appears as usual due to our setting $\eps_j=\eps (t_{j+1}-t_j)$ in each subinterval.
\end{proof}
\begin{proof}[Proof of Proposition~\ref{prop:LDP_dynamic_pointwise}]
In view of Lemma~\ref{lem:LDP_slicing} and by the Dawson-G\"artner theorem \cite[thm. 4.6.1]{dembo2009large}, the desired LDP will automatically hold with good rate function
 $$
 \sup\limits_{|\tau_N|\to 0} C_x^N(\o_{t_0},\dots,\o_{t_N})
 =
 \iota_{\{\o_{0}=x\}}+\sup\limits_{|\tau_N|\to 0} C^N(\o_{t_0},\dots,\o_{t_N}),
 \hspace{1cm}\o\in\O
 $$
 so clearly it is enough to show that
 $$
%  \sup\limits_{|\tau_N|\to 0}C^N(\o_{t_0},\dots,\o_{t_N})
%  =
 \sup \limits_{|\tau_N|\to 0}\sum\limits_{j=0}^{N-1}\frac{c(\o_{t_j},\o_{t_{j+1}})}{t_{j+1}-t_j}
%  =\uC(\o)
 =
 \begin{cases}
\int_0^1 \uL(\o_t,\dot\o_t)\rd t & \mbox{if }\o\in H^1\\
+\infty & \mbox{else}
 \end{cases}
 $$
 for any fixed $\o\in\O$.

 To this end, consider first $\o\in H^1$.
 Setting $\g^j_{s}\coloneqq \o_{t_j+s(t_{j+1}-t_j)}$, we have by \eqref{eq:c=min_L} and $2$-homogeneity of $\uL(x,\cdot)$ that
$$
c(\o_{t_j},\o_{t_{j+1}})
\leq \int_0^1 \uL(\g^j_s,\dot\g^j_s)\rd s =(t_{j+1}-t_j)\int_{t_j}^{t_{j+1}}\uL(\o_t,\dot\o_t)\rd t.
$$
As a consequence
$$
C^N(\o_{t_0},\dots,\o_{t_N})
=
\sum\limits_{j=0}^{N-1}\frac{c(\o_{t_j},\o_{t_{j+1}})}{t_{j+1}-t_j}
\leq \sum\limits_{j=0}^{N-1} \int_{t_j}^{t_{j+1}}\uL(\o_t,\dot\o_t)\rd t
=\int_0^1 \uL(\o_t,\dot\o_t)\rd t
$$
for any $N$.
In order to get a lower bound, let $(\o^N_t)_{t\in[0,1]}$ be the continuous, piecewise geodesic interpolation of $\o_{t_0},\dots,\o_{t_N}$, where geodesics are not the Euclidean ones but rather the intrinsic ones from Theorem~\ref{theo:c_static_dynamic-geodesics}, and appropriately scaled in time so that
$$
C^N(\o_{t_0},\dots,\o_{t_N})
=
\sum\limits_{j=0}^{N-1}\frac{c(\o_{t_j},\o_{t_{j+1}})}{t_{j+1}-t_j}= \int_0^1\uL(\o^N_t,\dot\o^N_t)\rd t.
$$
By coercivity \eqref{eq:L_coercive} and the previous upper bound we see that $\{\o^N\}_{N}$ is bounded in $H^1$ with fixed endpoints $\o^N_0=\o_0$ and $\o^N_1=\o_1$.
We can therefore assume that $\o^N$ has a weak-$H^1$ limit, and this limit is of course $\o$.
Hence by the weak lower semicontinuity in Theorem~\ref{theo:c_static_dynamic-geodesics} and from the previous upper bound we obtain
\begin{multline*}
\int_0^1\uL(\o_t,\dot\o_t)\rd t
\leq
\liminf\limits_{N\to\infty} \int_0^1\uL(\o_t^N,\dot\o_t^N)\rd t
\\
=
\liminf\limits_{N\to\infty} C^N(\o_{t_0},\dots,\o_{t_N})
\\
\leq
\limsup\limits_{N\to\infty} C^N(\o_{t_0},\dots,\o_{t_N})
\leq
\int_0^1\uL(\o_t,\dot\o_t)\rd t
\end{multline*}
as claimed.

If now $\o\not\in H^1$, standard properties of Sobolev functions give that $\sum\limits_{j=0}^{N-1}\frac{|\o_{t_{j+1}}-\o_{t_j}|^2}{t_{j+1}-t_j}\to+\infty$ as $|\tau^N|\to 0$.
The coercivity \eqref{eq:L_coercive} integrated along geodesics readily gives $c(x,y) \geq \frac{1}{2a}|x-y|^2$ for any $x,y$, thus
$$
\sum\limits_{j=0}^{N-1}\frac{c(\o_{t_{j+1}},\o_{t_j})}{t_{j+1}-t_j}
\geq
\frac{1}{2a}\sum\limits_{j=0}^{N-1}\frac{|\o_{t_{j+1}}-\o_{t_j}|^2}{t_{j+1}-t_j}\xrightarrow[N\to\infty]{}+\infty
$$
and the proof is complete.
\end{proof}
We are finally in position to establish our main result.
\begin{proof}[Proof of Theorem~\ref{theo:LDP_dynamic}]
Since the uniform topology is obviously stronger than pointwise convergence, by \cite[Corollary. 4.2.6]{dembo2009large} it is enough to prove that $\{R^\eps_x\}_{\eps>0}$ is exponentially tight for the uniform topology.
By \cite[Theorem 4.1]{Puhalskii} it is enough to check the ``exponential Arzel\`a-Ascoli'' conditions
\begin{enumerate}[(i)]
 \item
 compactness:
 $$
 \lim\limits_{M\to\infty}\limsup\limits_{\eps\to 0}\eps\log R^\eps_x\Big(\Big\{\o:\ |\o(0)|\geq M\Big\}\Big)=-\infty,
 $$
 \item
 \label{item:equicontinuiety_exp}
 equicontinuity: for all fixed $\eta >0$,
 \begin{equation}
\label{eq:expo_equicont}
 \lim\limits_{\delta\to 0}\limsup\limits_{\eps\to 0}\eps\log R^\eps_x\Big(\Big\{\o:\ \sup\limits_{|t-s|\leq \delta}|\o(t)-\o(s)|\geq \eta\Big\}\Big)=-\infty.
 \end{equation}
\end{enumerate}
Since $\o(0)=x$ for $R^\eps_x$-a.e. $\o\in\O$ the first condition is trivial so we only focus on the second part.
Fix $\eta>0$.
We first deal with the vertical component, which according to \eqref{eq:X'_distrib_sum_BM} reads
$$
{X^\eps_t}'
= \sqrt a \tilde B'_{O_{\eps t}}+ \check B'_{\eps t-O_{\eps t}}
$$
for two independent $(d-1)$-Brownian motions $\tilde B',\hat B'$.
For fixed $0\leq s\leq t \leq 1$ we have
$$
\left|{X^\eps_t}'-{X^\eps_s}'\right|
\leq
\sqrt a \left|\tilde B'_{O_{\eps t}}-\tilde B'_{O_{\eps s}}\right| + \left|\check B'_{\eps t-O_{\eps t}}-\check B'_{\eps s-O_{\eps s}}\right|
$$
and therefore
\begin{multline*}
 \PP
 \left(\sup\limits_{|t-s|\leq \delta}\left|{X^\eps_t}'-{X^\eps_s}'\right|\geq \eta\right)
\leq
\PP\left(\sup\limits_{|t-s|\leq \delta}|\tilde B'_{O_{\eps t}}-\tilde B'_{O_{\eps s}}|\geq \frac \eta{2\sqrt a}\right)
\\
+
\PP \left(\sup\limits_{|t-s|\leq \delta}|\hat B'_{\eps t-O_{\eps t}}-\hat B'_{\eps s-O_{\eps s}}|\geq \frac \eta 2\right)
\eqqcolon
A'_\eps +B'_\eps
\end{multline*}
For the first $A'_\eps$ term, recall that the occupation time $O_t=\int_0^t \1_{\{X^1_\tau=0\}}\rd \tau$ only depends on the horizontal sticky Brownian motion, which in turn is completely independent of the two vertical Brownian motions $\tilde B',\hat B'$.
As a consequence, by independence and standard properties of Brownian increments, $\tilde B_{O_{\eps t}}-\tilde B_{O_{\eps s}}$ is distributed as $\tilde B_{O_{\eps t}-O_{\eps s}}$ and therefore
$$
A'_\eps = \PP\left(\sup\limits_{|t-s|\leq \delta}|\tilde B'_{O_{\eps t}-O_{\eps s}}|\geq \frac \eta{2{\sqrt a}}\right).
$$
Now, since $0\leq O_{\eps t}-O_{\eps s}=\int_{\eps s}^{\eps t}\1_{\{X^1_\tau=0\}}\rd \tau \leq \eps (t-s) $ one has
$$
\sup\limits_{|t-s|\leq \delta}|\tilde B'_{O_{\eps t}-O_{\eps s}}|\leq
\sup \limits_{\tau \leq \eps \delta}|\tilde B'_\tau|
$$
almost surely.
Finally, by standard properties of Brownian motion one has the classical estimate (see e.g. \cite[Lemma 5.2.1]{dembo2009large}
$$
\PP \left(\sup _{\tau\in [0,T]}|B_\tau|\geq M\right) \leq 4(d-1) \exp\left(-\frac{M^2}{2(d-1) T}\right),
\hspace{1cm}\forall M>0,\, T>0.
$$
Putting $M=\eta/2{\sqrt a}$ and $T=\eps\delta$ we see that
$$
A'_\eps \leq 4(d-1)\exp\left( -\frac{\eta ^2}{8a(d-1)\eps\delta}\right)
$$
Similarly, since $t-O_t$ and $\hat B'$ are independent and $ 0\leq (\eps t-O_{\eps t})-(\eps s-O_{\eps s})\leq \eps(t-s)$, one easily gets that
$$
B'_\eps
\leq 4(d-1)\exp\left( -\frac{\eta ^2}{8 (d-1)\eps\delta}\right)
$$
and thus
$$
A'_\eps+B'_\eps \leq C\exp(-c\eta^2/\delta\eps)
$$
for some uniform constants $C,c>0$ independent of $\eta,\eps,\delta>0$.
This immediately gives the exponential equicontinuity \eqref{eq:expo_equicont} for the horizontal component.

Let us now turn to the horizontal component ${X^{\eps}_{t}}^1=X^1_{\eps t}$, where $(X^1_\tau)_{\tau\geq 0}$ is the one-dimensional sticky-reflecting Brownian motion started from $x_1\geq 0$.
First of all, by fine properties of the one-dimensional reflected SBM \cite[eqs. (2.11) and (3.7)]{engelbert2014stochastic} we have that
$$
X^1_t
=
\left|S_t\right|
=
\left|x_1+B_t -\int _0^t\1_{\{ S_\tau=0\}}\rd B^0_\tau\right|,
$$
where $S_t$ is the two-sided, non-reflected sticky Brownian motion on $\R$ and $B_t,B^0_t$ are two standard Brownian motions in $\R$.
It follows that
$$
|X^1_t-X^1_s|
\leq | S_r-S_s|
\leq \left|B_t-B_s\right| + \left|\int_s^t\1_{\{ S_\tau=0\}}\rd B^0_\tau\right|
$$
and therefore, for fixed $\eta,T>0$ and $s\geq 0$
\begin{align*}
\PP\left(\sup\limits_{t\in [s,s+T]} |S_t-S_s|\geq \eta\right)
& \leq
\PP\left(\sup\limits_{t\in [s,s+T]} |B_t-B_s|\geq \eta/2\right)
\\
&\hspace{2cm} +
\PP\left(\sup\limits_{t\in [s,s+T]} \left|\int_s^t \1_{\{ S_\tau=0\}}\rd B^0_\tau\right|\geq \eta/2\right)
\\
& \eqqcolon A^1_T+B^1_T.
\end{align*}
For the first term we have exactly as before
$$
A^1_T
=
\PP\left(\sup _{\tau\in [0,T]}|B_\tau|\geq \eta/2\right) \leq 4 \exp\left(-\frac{\eta^2}{8 T}\right).
$$
For the second term let us set
$$
Y_t\coloneqq \int_s^t\1_{\{S_\tau=0\}}\rd B^0_\tau.
$$
For fixed $\lambda>0$ let $f(y)\coloneqq \exp(\lambda y)$ and $Z_t\coloneqq f(Y_t)$.
Because $f$ is convex and $Y$ is a martingale, $Z$ is a submartingale and therefore by Doob's martingale inequality
$$
\PP\left(\sup\limits_{t\in [s,s+T]}  Y_t\geq \eta/2\right)
=
\PP\left(\sup\limits_{t\in [s,s+T]}  Z_t\geq f(\eta/2)\right)
\leq \frac{\EE(Z_{s+T})}{\exp(\lambda \eta/2)}.
$$
Applying It\^o's lemma
$$
\rd Z_t=f'(Y_t) \rd Y_t +\frac 12 f''(Y_t)\rd\langle Y\rangle_t
=\lambda Z_t \rd Y_t + \frac{\lambda^2}2 Z_t \1_{\{S_t=0\}}\rd t,
$$
taking expectation, and integrating we see that
$$
\EE(Z_{t})
=
\EE(Z_{s}) + \frac{\lambda^2} 2\EE\int_s^ t Z_\tau \1_{\{ S_\tau=0\}}\rd \tau
=
1 + \frac{\lambda^2} 2\int_s^ t\EE( Z_\tau \1_{\{S_\tau=0\}})\rd \tau
\leq
1 + \frac{\lambda^2} 2\int_s^ t\EE( Z_\tau )\rd \tau.
$$
Gr\"onwall's inequality yields $\EE(Z_{s+T})\leq \exp (\lambda^2 T/2)$, thus
$$
\PP\left(\sup\limits_{t\in [s,s+T]}  Y_t\geq M/2\right)
\leq \exp(\lambda^2T/2-\lambda \eta/2)
$$
for any $\lambda>0$.
The right-hand side is minimized for $\lambda=\frac \eta{2T}$ and therefore
$$
\PP\left(\sup\limits_{t\in [s,s+T]}  Y_t\geq \eta/2\right)
\leq \exp(-\eta^2/8T).
$$
% On can estimate exactly in the same way
% $$
% \PP\left(\inf\limits_{t\in [s,s+T]}  Y_t\leq -\eta/2\right)
% \leq \exp(-\eta^2/8T)
% $$
Since $Y_t$ is distributed as $-Y_t$ we conclude that
$$
B_T^1=\PP\left(\sup\limits_{t\in [s,s+T]}  |Y_t|\geq \eta/2\right)
\leq 2\exp(-\eta^2/8T).
$$
Finally scaling down time $t,s\leftrightarrow\eps s,\eps t$ and $T=\eps \delta$ gives
\begin{multline*}
\PP\left(\sup\limits_{|s-t|\leq \delta}|{X^{\eps}_t}^1-{X^{\eps}_s}^1|  \geq \eta\right)
=
\PP\left(\sup\limits_{|s-t|\leq \delta}|X^1_{\eps t}-X^1_{\eps s}|  \geq \eta\right)
\leq
A^1_{\eps \delta}+ B^1_{\eps\delta}
\leq C \exp(-c\eta^2/\eps\delta).
\end{multline*}
This shows that the horizontal component also satisfies the equicontinuity \eqref{eq:expo_equicont} and concludes the proof.
\end{proof}

\subsection*{Acknowledgments}
J.-B.C. was supported by FCT - Funda\c{c}\~{a}o para a Ci\^encia e a Tecnologia, under the project UIDB/04561/2020.
L.M. was funded by FCT - Funda\c{c}\~{a}o para a Ci\^encia e a Tecnologia through a personal grant 2020/00162/CEECIND (DOI 10.54499/2020.00162.CEECIND/CP1595/CT0008) and project UIDB/00208/2020 (DOI 10.54499/UIDB/00208/2020) and warmly thanks A. Baradat, C. L\'eonard, and M. von Renesse for discussions pertaining to this work. L.N. benefited from the support of the FMJH Program PGMO and from the ANR project GOTA (ANR-23-CE46-0001). 

%%%%%%%%%%%%%%%%%%%%%%%%%%%%%%%%%%%%%%%%%%%%%%%%%%%%%%%%%%%%%%%%%%%%%%%%%%%%%%%%%%%%

%%%%%%%%%%%%%%%%%%%%%%%%%%%%%%%%%%%%%%%%%%%%%%%%%%%%%%%%%%%%%%%%%%%%%%%%%%%%%%%%%%%%

\bibliographystyle{plain}
\bibliography{./biblio}

\bigskip
\noindent
{\sc
Jean-Baptiste Casteras (\href{mailto:jeanbaptiste.casteras@gmail.com}{\tt jeanbaptiste.casteras@gmail.com}).
\\
CMAFcIO, Faculdade de Ci\^encias da Universidade de Lisboa, Edificio C6, Piso 1, Campo Grande 1749-016 Lisboa, Portugal
}
\\

\noindent
{\sc
L\'eonard Monsaingeon (\href{mailto:leonard.monsaingeon@univ-lorraine.fr}{\tt leonard.monsaingeon@univ-lorraine.fr}).
\\
Institut \'Elie Cartan de Lorraine, Universit\'e de Lorraine, Site de Nancy B.P. 70239, F-54506 Vandoeuvre-l\`es-Nancy Cedex, France
\\
Group of Mathematical Physics, Departamento de Matem\'atica, Instituto Superior T\'ecnico, Av. Rovisco Pais, 1049-001 Lisboa, Portugal
}
\\

\noindent
{\sc
Luca Nenna (\href{mailto:luca.nenna@universite-paris-saclay.fr}{\tt luca.nenna@universite-paris-saclay.fr}).
\\
Universit\'e Paris-Saclay, CNRS, Laboratoire de math\'ematiques d'Orsay, ParMA, Inria Saclay, 91405, Orsay, France
}
\end{document}